\documentclass[hidelinks,onefignum,onetabnum]{siamart251104}

\usepackage{stmaryrd}
  
\newcommand{\mc}[1]{\multicolumn{1}{c}{#1}}

\usepackage{amsfonts}
\usepackage{graphicx}
\usepackage{epstopdf}
\ifpdf
  \DeclareGraphicsExtensions{.eps,.pdf,.png,.jpg}
\else
  \DeclareGraphicsExtensions{.eps}
\fi

\newsiamremark{remark}{Remark}
\newsiamremark{hypothesis}{Hypothesis}
\crefname{hypothesis}{Hypothesis}{Hypotheses}
\newsiamthm{claim}{Claim}
\newsiamremark{fact}{Fact}
\crefname{fact}{Fact}{Facts}
\newsiamremark{problem}{Problem}

\headers{Monolithic Newton-Krylov Solver for Space-Time FEM}{
  Nils Margenberg, Markus Bause}

\title{
A Monolithic \texorpdfstring{$\lowercase{\boldsymbol{hp}}$}{hp} Space-Time Multigrid Preconditioned Newton-Krylov Solver for Space-Time FEM applied to the Incompressible Navier-Stokes Equations\thanks{Submitted to the editors DATE.}}

\author{
  Nils Margenberg\thanks{
    University of Magdeburg,
    Institute for Analysis and Numerics,
    Universit\"atsplatz 2,
    39104 Magdeburg,
    Germany,
    \texttt{nils.margenberg@ovgu.de}}
    \and
    Markus Bause\thanks{Helmut Schmidt University,
    Faculty of Mechanical and Civil Engineering,
    Holstenhofweg 85,
    22043 Hamburg,
    Germany}
    }

\usepackage{amsopn}
\usepackage{amsmath,amssymb,mathtools}

\providecommand{\R}{\mathbb{R}}

\usepackage{tikz}
\usetikzlibrary{3d,decorations.pathreplacing,shapes, positioning, fit, quotes,angles, fit, calc, patterns, external}
\usetikzlibrary{spy,shadows}
\tikzset{
  spy using overlaysshadow/.style={
    spy scope={#1,
         every spy on node/.style={
            circle,
            fill, fill opacity=0.25, text opacity=1
             },
         every spy in node/.style={
                 circle, circular drop shadow,
                 fill=white, draw, cap=round
            }
        }
    }
}
\usepackage{pgfplots}
\usepgfplotslibrary{groupplots,colorbrewer,fillbetween}
\pgfplotsset{compat=newest}%
\pgfplotsset{ colormap/Set1-4, cycle multiindex* list={ mark
    list*\nextlist Set1-4\nextlist }, every axis/.append style = {thick},%
}%
\pgfkeys{/pgf/number format/.cd,1000 sep={\,}}
\usepgfplotslibrary{colormaps}
\pgfplotsset{ every mark/.append style={4pt},tick style = {thick,black}}%
\usepackage{diagbox}
\newcommand{\N}{\mathbb{N}} %
\renewcommand{\d}{\mkern3mu\text{d}} %
\usepackage{MnSymbol}
\usepackage[normalem]{ulem}
\usepackage{multirow}
\ifpdf
\hypersetup{
  pdftitle={A Monolithic Newton–Krylov Solver with hp Space–Time Multigrid Preconditioning for Tensor-Product Space–Time Finite-Element Discretizations of the Incompressible Navier–Stokes Equations},
  pdfauthor={Margenberg, N.}
}
\fi

\usepackage{tikz}
\usetikzlibrary{3d,decorations.pathreplacing,shapes, positioning, fit, quotes,angles, fit, calc, patterns}
\usetikzlibrary{spy,shadows}
\tikzset{
  spy using overlaysshadow/.style={
    spy scope={#1,
      every spy on node/.style={
        circle,
        fill, fill opacity=0.25, text opacity=1
      },
      every spy in node/.style={
        circle, circular drop shadow,
        fill=white, draw, cap=round
      }
    }
  }
}
\usepackage{pgfplots}
\usepgfplotslibrary{groupplots,colorbrewer,fillbetween}
\pgfplotsset{compat=newest}%
\pgfplotsset{ colormap/Set1-4, cycle multiindex* list={ mark
    list*\nextlist Set1-4\nextlist }, every axis/.append style = {thick},%
}%
\pgfkeys{/pgf/number format/.cd,1000 sep={\,}}
\usepgfplotslibrary{colormaps}
\pgfplotsset{ every mark/.append style={4pt},tick style = {thick,black}}%
\usepackage{booktabs}
\usepackage{enumitem}
\usepackage{siunitx}
\usepackage{fp}
\newcommand{\eval}[2][]{%
  \ifthenelse{\equal{#1}{}}{%
    \FPeval{\result}{#2}%
    \num[round-precision=4]{\result}%
  }{%
    \FPeval{\result}{#2}%
    \num[scientific-notation=false,round-mode=places,print-zero-exponent=false,tight-spacing=true,round-precision=#1]{\result}%
  }%
}
\usepackage[noend,algoruled,norelsize,linesnumbered,lined,commentsnumbered,algo2e]{algorithm2e}
\usepackage{subcaption}
\definecolor{myblue}{RGB}{0 83 139}
\definecolor{myred}{RGB}{114 16 69}
\definecolor{mygreen}{RGB}{0 94 0}

\makeatletter
\newcommand{\customlabel}[2]{%
  \protected@write \@auxout {}{\string \newlabel {#1}{{#2}{\thepage}{#2}{#1}{}} }%
  \hypertarget{#1}{#2}%
}
\makeatother

\begin{document}

  \maketitle
\begin{abstract}
We present a monolithic \(hp\) space-time multigrid method (\(hp\)-STMG) for
tensor-product space-time finite element discretizations of the incompressible
Navier--Stokes equations. We employ mapped inf-sup stable pairs
\(\mathbb Q_{r+1}/\mathbb P_{r}^{\mathrm{disc}}\) in space and a slabwise
discontinuous Galerkin DG\((k)\) discretization in time. The resulting fully
coupled nonlinear systems are solved by Newton--GMRES preconditioned with
\(hp\)-STMG, combining geometric coarsening in space with polynomial coarsening
in space and time. Our main contribution is an \(hp\)-robust and practically
efficient extension of space-time multigrid to Navier--Stokes: matrix-free
operator evaluation is retained via column-wise, state-dependent spatial
kernels; the nonlinear convective term is handled by a reduced, order-preserving time quadrature. Robustness is ensured
by an inexact space-time Vanka smoother based on patch models with single time point evaluation. The method is implemented
in the matrix-free multigrid framework of \texttt{deal.II} and demonstrates
\(h\)- and \(p\)-robust convergence with robust solver performance across a
range of Reynolds numbers, as well as high throughput in large-scale
MPI-parallel experiments with more than \(10^{12}\) degrees of freedom.
\end{abstract}

  \section{Introduction}
  Accurate and efficient simulation of incompressible flows remains challenging,
particularly in the high-resolution, high-order regime. Space-time finite element
methods (STFEMs) provide a natural framework for parallelism in space and time by
treating time as an additional coordinate and enabling a unified variational
discretization. This work transfers our space-time multigrid methodology for the
instationary Stokes equations~\cite{MaMuBa25} to the nonlinear Navier--Stokes case. Two key ingredients make this extension feasible in a matrix-free setting:
(i) a \emph{tailored reduced time quadrature} for nonlinear space-time terms to
reduce the cost per residual/Jacobian application while preserving the overall
temporal order, and (ii) an \emph{inexact} space-time Vanka smoother based on a
coefficient patch model evaluated at a single time point, enabling reuse
of local factorizations. We prove that lower-order quadrature for the nonlinear
triple products preserves the discretization order, and that the inexact
smoother remains uniformly close to the exact Vanka smoother. Numerically, this
controlled inexactness does not degrade Newton--Krylov convergence or multigrid
robustness.

We consider slab-wise tensor-product space-time discretizations. In space, we use
mapped, inf-sup stable $\mathbb Q_{r+1}/\mathbb P_{r}^{\mathrm{disc}}$ pairs with
$r\in\mathbb N$, and in time a discontinuous Galerkin method DG$(k)$ of order
$k\in\mathbb N_0$. {Continuous-in-time Galerkin discretizations are not pursued
here, since a satisfactory definition of a continuous pressure trajectory
remains an open issue; cf.~\cite{anselmannOptimalorderPressureApproximation2025}.
The difficulty is that the initial pressure value is defined implicitly by the
Navier--Stokes equations.} The resulting nonlinear space-time system is solved by
Newton--GMRES iterations preconditioned by an \(hp\) space-time multigrid method
(\(hp\) STMG), combining geometric and polynomial coarsening in space and time. {Here, $p$-multigrid refers to coarsening in polynomial degree,
while $h$-multigrid (geometric multigrid) refers to mesh coarsening.
In our implementation, $h$-coarsening is restricted to space due to the
time-marching realization; temporal $h$-coarsening is deferred to
future work.} Our implementation is based on the matrix-free multigrid framework in
\texttt{deal.II}~\cite{africa_dealii_2024,kronbichlerGenericInterfaceParallel2012,munchEfficientDistributedMatrixfree2023,fehnHybridMultigridMethods2020}.
The source code is available~\cite{margenberg_monolithic_2026}.

Matrix-free \emph{monolithic} multigrid methods for Stokes-type systems are a prototypical building block for high-order incompressible flow
solvers \cite{kohlTextbookEfficiencyMassively2022,jodlbauerMatrixfreeMonolithicMultigrid2024}.
Monolithic $ph$-multigrid preconditioners have been studied for high-order
stationary Stokes~\cite{voroninMonolithicMultigridPreconditioners2024}.
Matrix-free Newton--Krylov solvers preconditioned by monolithic multigrid have
also been demonstrated for stabilized incompressible Navier--Stokes
\cite{prietosaavedraMatrixFreeStabilizedSolver2024}, and high-order space-time
DG formulations for incompressible Navier--Stokes are available
\cite{rhebergenCockburnVegtSTDG2013}. Slab-wise tensor-product space-time
discretizations with Newton solvers have been investigated
\cite{rothTensorProductSpaceTimeGoalOriented2023}. The present work targets
robust, scalable \(hp\) multigrid preconditioning for nonlinear tensor-product
space-time finite element Navier--Stokes systems.

  Parallel-in-time integration methods address the sequential bottleneck of
  classical time stepping; see~\cite{ganderTimeParallelTime2024}. Fully implicit
  Runge--Kutta (FIRK) methods offer a complementary route to high-order time
  integration but lead to large coupled stage systems, closely related to
  tensor-product STFEM.\@ Stage-parallel and transformed solvers for FIRK
  systems in DG-based flow simulations are developed
  in~\cite{paznerPerssonStageParallelIRK2017}, and Vanka-based monolithic
  multigrid for implicit Runge--Kutta discretizations of incompressible flow is
  studied in~\cite{abu-labdeh_monolithic_2023}. Further developments for FIRK
  systems with nonlinearities and multilevel block solvers are addressed in~\cite{southworthKrzysikPaznerFIRKII2022,munchStageParallelFullyImplicit2023}.

  We retain the STFEM viewpoint and show that
  the space-time multigrid methodology developed
  in~\cite{MaMuBa25,margenbergSpaceTimeMultigridMethod2024a} can be made practical
  for Navier--Stokes via temporal underintegration and an inexact space-time Vanka
  smoother. The choice of the smoother is motivated by the proven effectiveness of
  Vanka-type smoothers in fluid mechanics~\cite{ahmedAssessmentSolversSaddle2018,anselmannGeometricMultigridMethod2023,MaMuBa25},
  coupled multiphysics~\cite{failerParallelNewtonMultigrid2021,anselmannEnergyefficientGMRESMultigrid2024} and acoustic
  wave equations~\cite{margenbergSpaceTimeMultigridMethod2024a}.

  This paper is organized as follows. \Cref{sec:stfem} introduces the continuous
  problem and the space-time finite element discretization. The
  resulting nonlinear algebraic system is derived in \Cref{sec:AlgSys}.
  \Cref{sec:nonlinear-solver} presents the \(hp\)-STMG preconditioned Newton--GMRES
  solver used throughout. Numerical experiments are reported in
  \Cref{sec:experiments}. \Cref{sec:conclusions} concludes with an evaluation of
  the results and an outlook.

  \section{\label{sec:stfem}Continuous and discrete problem}
  \subsection{Continuous problem}
  We consider the incompressible, nonstationary\linebreak[4] Navier--Stokes system
  \begin{subequations}\label{eq:NSE}
    \begin{alignat}{3}
      \label{eq:NSE_1}
      \partial_t \boldsymbol{v} + (\boldsymbol v\cdot \boldsymbol\nabla) \boldsymbol v - \nu  \boldsymbol\Delta \boldsymbol{v} + \boldsymbol\nabla p & = \boldsymbol{f}  && \quad \text{in }\; \Omega \times (0,\,T)\,,\\
      \label{eq:NSE_2}
      \boldsymbol\nabla \cdot \boldsymbol{v}  & = 0 && \quad \text{in }\; \Omega \times (0,\,T)\,,\\
      \label{eq:NSE_3}
      \boldsymbol{v}(0) & = \boldsymbol{v}_0 && \quad \text{in }\; \Omega\,,\\
      \label{eq:NSE_4}
      \boldsymbol{v}  & = \boldsymbol{g}_D  && \quad \text{on } \; \Gamma_D \times (0,\,T)\,,\\
      \label{eq:NSE_5}
      (\nu \boldsymbol\nabla \boldsymbol v - p \boldsymbol I) \boldsymbol n & = \boldsymbol 0 && \quad \text{on } \; \Gamma_N \times (0,\,T)\,.
    \end{alignat}
  \end{subequations}
  Here, $\Omega\subset\mathbb R^d$, $d\in\{2,3\}$, is a bounded Lipschitz domain and $T>0$ the final time. We split
  $\Gamma\coloneqq\partial\Omega$ into $\Gamma_D$ and $\Gamma_N$, i.e.\
  $\Gamma=\overline{\Gamma_D}\cup\overline{\Gamma_N}$, with
  $|\Gamma_D|_{d-1}>0$, $|\Gamma_N|_{d-1}>0$, and $\Gamma_D\cap\Gamma_N=\emptyset$.
  The outward unit normal of $\Gamma$ is $\boldsymbol n=\boldsymbol n(\boldsymbol x)$ and
  $\boldsymbol I$ denotes the identity matrix. The unknowns in~\eqref{eq:NSE} are
  the velocity field $\boldsymbol v$ and pressure $p$, while $\boldsymbol f$,
  $\boldsymbol g_D$, and $\boldsymbol v_0$ are given, sufficiently smooth data.
  The kinematic viscosity satisfies $\nu>0$. {We assume sufficient regularity
  of solutions to~\eqref{eq:NSE}, in particular up to $t=0$, to justify higher-order space-time approximations.}

  We use standard notation. $H^m(\Omega)$ denotes the Sobolev space of $L^2(\Omega)$ functions with (weak) derivatives up to order $m$ in $L^2(\Omega)$. The $L^2(\Omega)$ inner product (and its vector- and tensor-valued counterparts) is denoted by $\langle\cdot,\cdot\rangle$, with boundary pairing $\langle\cdot,\cdot\rangle_\Gamma$ and $\Gamma \in \{\partial \Omega, \Gamma_D\}$. Further, we use $\partial_n \boldsymbol w\coloneq ( \boldsymbol\nabla \boldsymbol
  w) \boldsymbol n$. We define
  \[
  \boldsymbol V\coloneq H^1(\Omega)^d,\quad Q\coloneq L^2(\Omega),\quad \boldsymbol V^{\operatorname{div}}\coloneq\{\boldsymbol v\in\boldsymbol V\mid \langle \boldsymbol\nabla \cdot \boldsymbol v, q\rangle=0\ \forall q\in Q\}.
  \]

\begin{remark}[Non mixed-type boundary condition]
If Dirchlet boundary conditions are prescribed on $\partial \Omega$ only, i.e.\ $\partial\Omega = \Gamma_D$, the pressure variable is determined uniquely up to a constant $c\in \R$ only. In this case, the pressure space is chosen as $Q(\Omega)\coloneq L^2(\Omega)$ with \(
L^2_0(\Omega)\coloneq\{ q\in L^2(\Omega)\mid \int_\Omega q\d  x=0\}
\) instead of $L^2(\Omega)$ such that its uniqueness is ensured. Further, we refer to \cite{MaMuBa25} for the modifications that have to be made for $\Gamma = \Gamma_D$ in the construction of the multigrid method.
\end{remark}

For the weak formulation of~\eqref{eq:NSE}, let $\boldsymbol X \coloneqq \boldsymbol V \times Q$. The semilinear form $A: \boldsymbol X\times \boldsymbol X \to \R$ with convective form of the first-order term in~\eqref{eq:NSE_1} is given by
\begin{equation}
  \label{eq:ns-form}
  \begin{aligned}
  A(\boldsymbol u)(\boldsymbol w) & \coloneq
  \langle (\boldsymbol v \cdot \nabla) \boldsymbol v, \boldsymbol z \rangle +  \langle\nu \boldsymbol\nabla \boldsymbol v - p\boldsymbol I, \nabla \boldsymbol z\rangle + \langle \boldsymbol\nabla \cdot \boldsymbol v,\, q\rangle
  \end{aligned}
\end{equation}
for $\boldsymbol u, \boldsymbol w \in \boldsymbol X$, with $ \boldsymbol u = (\boldsymbol v,p)$ and $ \boldsymbol w = (\boldsymbol z,q)$. Well-definedness of $A$ on $\boldsymbol X\times \boldsymbol X $ is ensured. For our unified tensor product approach, we rewrite the convective term by using that $(\boldsymbol v \cdot \nabla) \boldsymbol v = \nabla \cdot (\boldsymbol v \otimes \boldsymbol v)$, if $\nabla \cdot \boldsymbol v = 0$. The tensor product of two vectors $\boldsymbol a, \boldsymbol b \in \R^d$ is defined by $\boldsymbol a \otimes \boldsymbol b = \boldsymbol a  \boldsymbol b^\top$.  This semilinear form $A^{tp}: \boldsymbol X\times \boldsymbol X \to \R$ with the convective term in divergence form is then given by
\begin{equation}
  \label{eq:ns-form-tp}
  \begin{aligned}
    A^{tp}(\boldsymbol u)(\boldsymbol w) & \coloneq
    \langle \nu \boldsymbol\nabla \boldsymbol v - p\boldsymbol I -\boldsymbol v\otimes\boldsymbol v,\nabla \boldsymbol z\rangle + \langle \boldsymbol\nabla \cdot \boldsymbol v,\, q\rangle
  \end{aligned}
\end{equation}
for $\boldsymbol u, \boldsymbol w \in \boldsymbol X$. Well-definedness of~\eqref{eq:ns-form-tp} is ensured. The semilinear forms comprise volume integrals only. Boundary integrals from integration by parts are added below.


\subsection{Space-time finite element discretization}

The construction principle for the space-time finite element spaces explores the so-called tensor product of Hilbert spaces. It can be rooted on the concept of multi-linear forms of Hilbert spaces. The product form of the space-time finite element spaces is strongly exploited in the construction and efficient implementation of our multigrid preconditioner.

\paragraph*{Time mesh and spaces} Let $I\coloneq(0,T]$ be partioned into $N$ subintervals $I_n\coloneq(t_{n-1},t_n]$, $n=1,\dots,N$, where $\tau = \max\{\tau_n \mid N=1,\ldots, N\}$. We denote the time mesh by $\mathcal M_\tau\coloneq\{I_1,\dots,I_N\}$. For $k\in\mathbb N_0$, let $\mathbb P_k(J;\mathbb R)$ be the space of piece-wise polynomials with maximum degree $k$ on $J\subset I$. We put
\begin{equation}
  \label{eq:DefYtau}
  Y_\tau^k(I)\coloneq\{ w_\tau:I\to\mathbb R\mid
  (w_\tau)_{|I_n}\in\mathbb P_k(I_n;\mathbb R)\;\forall I_n\in\mathcal
  M_\tau\}\,.
\end{equation}

\paragraph*{Spatial mesh and spaces} Let $\mathcal T_h$ be a shape-regular
quadrilateral/hexahedral mesh of $\Omega$ with mesh size $h>0$. On $K\in\mathcal T_h$ we put, for fixed $r\in \N$,
\begin{equation}\label{Def:VKQK}
  \boldsymbol V^{r+1}(K)\coloneq (\mathbb Q_{r+1})^d\circ\boldsymbol T_K^{-1}\,,\qquad
  Q^r(K)\coloneq \mathbb P_r^{\mathrm{disc}}\circ\boldsymbol T_K^{-1}\,,
\end{equation}
where $\boldsymbol T_K$ is the standard multilinear map from the reference element to $K$. We employ the mapped variant of $\mathbb P_r^{\mathrm{disc}}$ for geometric consistency on curved/non-affine meshes and improved conditioning of the algebraic systems. Using~\eqref{Def:VKQK}, we define
\begin{subequations}\label{Def:VhQh}
  \begin{align}
    \boldsymbol V_h^{r+1}(\Omega) &\coloneq \{\boldsymbol v_h\in \boldsymbol V: (\boldsymbol v_h)_{|K}\in\boldsymbol V_{r+1}(K)\ \forall K\in\mathcal T_h\}\,,\\
    Q_h^{r}(\Omega) &\coloneq \{ q_h\in Q: (q_h)_{|K}\in Q_r(K)\ \forall K\in\mathcal T_h\}\,.
  \end{align}
\end{subequations}
The subspace of $\boldsymbol V_h$ of discretely divergence-free functions is
\begin{equation}\label{Def:Vdiv}
  \boldsymbol V_h^{\operatorname{div}}(\Omega)\coloneq\{\boldsymbol
  v_h\in\boldsymbol V_h^{r+1}(\Omega)\mid \langle\boldsymbol\nabla \cdot \boldsymbol
  v_h,\,q_h\rangle=0\ \forall q_h\in Q_h^{r}(\Omega)\}\,.
\end{equation}

\paragraph*{Space-time tensor product spaces} The global fully discrete spaces are the
algebraic tensor products
\begin{equation}
  \label{eq:GDS}
  \boldsymbol H_{\tau h}^{\boldsymbol v} \coloneq Y_\tau^k(I)\otimes \boldsymbol V_h^{r+1}(\Omega)\,,\qquad
  H_{\tau h}^{p} \coloneq Y_\tau^k(I)\otimes Q_h^{r}(\Omega)\,,\qquad \boldsymbol{X}_{\tau h}^{kr} \coloneqq \boldsymbol H_{\tau h}^{\boldsymbol v} \times H_{\tau h}^{p} \,.
\end{equation}

\begin{remark}[Tensor products and Bochner spaces]
  The algebraic tensor product $Y_\tau(I)\otimes V_h(\Omega)$ is the span of
  separable functions $f\otimes g:(t,\boldsymbol x)\mapsto f(t)g(\boldsymbol x)$
  with $f\in Y_\tau(I)$ and $g\in V_h(\Omega)$; see
  \cite[Section~1.2.3]{picard_partial_2011}. The Hilbert spaces in~\eqref{eq:GDS}
  are isometric to the Bochner spaces $Y_\tau^k(I;\boldsymbol V_h^{r+1}(\Omega))$
  and $Y_\tau^k(I;Q_h^{r}(\Omega))$ of piece-wise polynomial functions with values in $\boldsymbol V_h^{r+1}$ and $Q_h^{r}$, respectively; cf.\
  \cite[Prop.~1.2.28]{picard_partial_2011}.
\end{remark}

\paragraph*{Right limits and jumps in time} For any piece-wise smooth $w:I\to B$ with respect to $\mathcal M_\tau$ (for instance, $w\in \boldsymbol H_{\tau h}^{\boldsymbol v}$), we define the right limit $w^+(t_n)\coloneq\lim_{t\to _n+0}w(t)$, $0\le n<N$, and the jump $\llbracket w\rrbracket_n\coloneq w^+(t_n)-w(t_n)$.

\paragraph*{Nitsche imposition on $\Gamma_D$} Instead of enforcing Dirichlet boundary conditions by the definition of the solution space, we apply Nitsche's method that imposes Dirichlet boundary conditions in a weak form. This offers appreciable advantages in the implementation of our framework. For this, let the positive and negative parts of $y\in \R$ be denoted by $y^\oplus \coloneqq \frac{1}{2}(|y|+y)$ and $y^\ominus \coloneqq \frac{1}{2}(|y|-y)$.

Firstly, we address the case in that the convective term is written as $(\boldsymbol v\cdot \nabla)  \boldsymbol v$. This is done in order to develop our approach for the tensor product from of the convection term. For brevity, we put $\boldsymbol X_h^r \coloneqq \boldsymbol V_h^{r+1}(\Omega)\times Q_h^r$. We let $\gamma_1 >0$ and $\gamma_2 >0$ denote two algorithmic parameters. Their choice is addressed below. The semilinear form $B_\gamma : \boldsymbol V \times \boldsymbol X_h \to \R $ is defined, for $v\in \boldsymbol V$ and $\boldsymbol w_h \in \boldsymbol X_h$, by
\begin{subequations}
\label{eq:def-b-gam}
\begin{alignat}{2}
\label{eq:def-b-gam-1}
B_\gamma (\boldsymbol v, \boldsymbol w_h) & \coloneqq B^{c} (\boldsymbol v, \boldsymbol w_h) + B^{s} (\boldsymbol v, \boldsymbol w_h) + B_\gamma^{r} (\boldsymbol v, \boldsymbol w_h)\,,\\
\label{eq:def-b-gam-2}
B^{c} (\boldsymbol v, \boldsymbol w_h) & \coloneqq  - \langle (\boldsymbol v \cdot \boldsymbol n)^\ominus \boldsymbol v , \boldsymbol z_h\rangle_{\Gamma_D}\,,\\
\label{eq:def-b-gam-2.5}
B^{s} (\boldsymbol v, \boldsymbol w_h) & \coloneqq - \langle \boldsymbol v, (\nu \nabla \boldsymbol z_h + q_h \boldsymbol I)\boldsymbol n \rangle_{\Gamma_D}\\
\label{eq:def-b-gam-3}
B_\gamma^r (\boldsymbol v, \boldsymbol w_h) & \coloneqq \nu\,  {\gamma_1}\,{h_{\Gamma_D}^{-1}} \langle \boldsymbol v, \boldsymbol z_h\rangle_{\Gamma_D} + {\gamma_2}\,{h_{\Gamma_D}^{-1}}\langle \boldsymbol v \cdot \boldsymbol n, \boldsymbol z_h\cdot \boldsymbol n \rangle_{\Gamma_D} \,.
\end{alignat}
\end{subequations}
Along with~\eqref{eq:ns-form}, we let $A_\gamma : \boldsymbol X_h \times \boldsymbol X_h \to \R$, for $\boldsymbol v_h, \boldsymbol w_h \in \boldsymbol X_h$, be given by
\begin{equation}
\label{eq:def-a-gam}
A_\gamma (\boldsymbol u_h)(\boldsymbol w_h) \coloneqq A (\boldsymbol u_h)(\boldsymbol w_h) - \langle (\nu \nabla \boldsymbol v_h - p_h \boldsymbol I)\boldsymbol n, \boldsymbol z_h \rangle_{\Gamma_D} + B_\gamma (\boldsymbol v_h,\boldsymbol w_h)\,.
\end{equation}
The terms~\eqref{eq:def-b-gam} and~\eqref{eq:def-a-gam} have the following explanations. The second term on the right hand side of~\eqref{eq:def-a-gam} is due to the application of integration by parts and substracts the natural boundary condition on $\Gamma_D$. The term in~\eqref{eq:def-b-gam-2} reflects the inflow boundary conditions. The term in~\eqref{eq:def-b-gam-2.5} is used to preserve the symmetry properties of the continuous system. The last two terms in~\eqref{eq:def-b-gam-3} are penalizations to ensure the stability of the discrete system. Together, the boundary pairings in~\eqref{eq:def-b-gam} model the effective Dirichlet conditions in the three different flow regimes of viscous effects $\boldsymbol v = \boldsymbol g$, convective behavior $(\boldsymbol v \cdot \boldsymbol n)^\ominus \boldsymbol v  = (\boldsymbol v \cdot \boldsymbol n)^\ominus \boldsymbol g$ and inviscid limit ($\boldsymbol v \cdot \boldsymbol n = \boldsymbol g \cdot \boldsymbol n$).

Next, we consider the divergence form $\nabla \cdot (\boldsymbol v \otimes \boldsymbol v)$ of the convective term. Along with~\eqref{eq:ns-form}, we let $A_\gamma ^{tp}: \boldsymbol X_h \times \boldsymbol X_h \to \R$, for $\boldsymbol v_h, \boldsymbol w_h \in \boldsymbol X_h$, be given by
 \begin{equation}
  \label{eq:def-a-gam-tp}
  \begin{aligned}
  A_\gamma^{tp} (\boldsymbol u_h)(\boldsymbol w_h) & \coloneqq A^{tp} (\boldsymbol u_h)(\boldsymbol w_h) - \langle (\nu \nabla \boldsymbol v_h - p_h \boldsymbol I)\boldsymbol n, \boldsymbol z_h \rangle_{\Gamma_D}\\
  & \quad  + \langle (\boldsymbol v_h \cdot \boldsymbol n) \boldsymbol v_h, \boldsymbol z_h \rangle_\Gamma + B_\gamma (\boldsymbol v_h, \boldsymbol w_h)\,.
  \end{aligned}
 \end{equation}
 The third term on the right-hand side of~\eqref{eq:def-a-gam-tp} is due to the application of integration by parts to the tensor product form of the convective term and ensures consistency. In~\eqref{eq:def-a-gam-tp}, we have with $y = y^\oplus - y^\ominus$ that
\begin{align*}
\langle (\boldsymbol v_h \cdot \boldsymbol n) \boldsymbol v_h, \boldsymbol z_h \rangle_\Gamma  & -  \langle (\boldsymbol v \cdot \boldsymbol n)^\ominus \boldsymbol v , \boldsymbol z_h\rangle_{\Gamma_D}\\
&= \langle (\boldsymbol v \cdot \boldsymbol n)^\oplus\boldsymbol v , \boldsymbol z_h\rangle_{\Gamma_D} - 2 \langle (\boldsymbol v \cdot \boldsymbol n)^\ominus \boldsymbol v , \boldsymbol z_h\rangle_{\Gamma_D} + \langle (\boldsymbol v \cdot \boldsymbol n) \boldsymbol v , \boldsymbol z_h\rangle_{\Gamma_N}\,.
\end{align*}
In flow problems that are of interest in practice, the Dirichlet boundary portion $\Gamma_D$ models an inflow boundary or fixed walls with no-slip condition. Then, $(\boldsymbol v \cdot \boldsymbol n)^\oplus =0$ is satisfied on $\Gamma_D$. Similarly, on the portion $\Gamma_N$ the property $(\boldsymbol v \cdot \boldsymbol n)^\minus =0$ is satisfied. Together, this implies that
\begin{equation}
\label{eq:bttp}
  \langle (\boldsymbol v_h \cdot \boldsymbol n) \boldsymbol v_h, \boldsymbol z_h \rangle_\Gamma  -  \langle (\boldsymbol v \cdot \boldsymbol n)^\ominus \boldsymbol v , \boldsymbol z_h\rangle_{\Gamma_D} = - 2 \langle (\boldsymbol v \cdot \boldsymbol n)^\ominus \boldsymbol v , \boldsymbol z_h\rangle_{\Gamma_D} + \langle (\boldsymbol v \cdot \boldsymbol n)^\oplus \boldsymbol v , \boldsymbol z_h\rangle_{\Gamma_N}\,.
\end{equation}
The terms on the right-hand side of~\eqref{eq:bttp} add further nonlinearities to the variational formulation. In the case of low Reynolds number flow, these quantities are small and might be neglected in computations. Further, for $\boldsymbol z_h = \boldsymbol v_h$ their nonnegativity is ensured, such that stability properties are not perturbed or weakened by the terms.

\paragraph*{Fully discrete problem} We are now in a position to define our tensor-product space-time finite element approximation of the Navier--Stokes system~\eqref{eq:NSE}. For the time discretization the discontinuous Galerkin method is applied. 

\begin{problem}[Discrete space-time variational problem]
  \label{Prob:DVP}
  Let $\boldsymbol f \in L^2(I;\boldsymbol H^{-1}(\Omega))$ and $\boldsymbol g_D \in L^2(I;\boldsymbol H^{1/2}(\Gamma_D))$
  be given. Let $\boldsymbol v_{0,h}\in\boldsymbol V_h^{\operatorname{div}}(\Omega)$ denote an approximation of $\boldsymbol v_0\in\boldsymbol
  V^{\operatorname{div}}(\Omega)$. Find $\boldsymbol u_{\tau h}=(\boldsymbol v_{\tau h},p_{\tau,h})\in \boldsymbol{X}_{\tau h}^{kr}$, with $\boldsymbol v_{\tau h}(0) \coloneqq \boldsymbol v_{0,h}$, such that for all $\boldsymbol w_{\tau h}= (\boldsymbol z_{\tau h},q_{\tau,h})\in \boldsymbol X_{\tau h}^{kr}$,
  \begin{equation}
    \label{eq:DNSE}
    \begin{aligned}
      \sum_{n=1}^{N} \int_{t_{n-1}}^{t_n}  & \langle \partial_t \boldsymbol v_{\tau h},\,\boldsymbol z_{\tau h}\rangle + A_\gamma^{tp}(\boldsymbol u_{\tau h},\boldsymbol w_{\tau h})\d  t + \sum_{n=0}^{N-1} \langle \llbracket \boldsymbol v_{\tau h}\rrbracket_n,\boldsymbol z_{\tau h}^+(t_n)\rangle \\
      &= \sum_{n=1}^{N} \int_{t_{n-1}}^{t_n}  \langle \boldsymbol f,\boldsymbol z_{\tau h}\rangle\d  t
      + \sum_{n=1}^{N} \int_{t_{n-1}}^{t_n}  B_\gamma (\boldsymbol g,\boldsymbol w_{\tau h})\d  t\,.
    \end{aligned}
  \end{equation}
\end{problem}

\begin{remark}[Preservation of tensor-product structure]
  In~\eqref{eq:DNSE}, all linear space-time terms can be written as
  algebraic tensor products of purely time and space dependent contributions; cf.~\Cref{sec:AlgSys}.
  In contrast to the linear parts, the convective terms depend on the
  velocity itself. This destroys the strict separability of time and space. Nevertheless,
  the associated spatial operators can be built as algebraic tensor products, depending on
  the velocity as time dependent coefficient function. The contribution of the convective terms and their Jacobian
  matrices remain amenable to matrix-free, sum-factorized evaluation, cf.~\Cref{sec:mf}.
  For the construction of the matrix based smoother, we construct a separable surrogate
  by evaluating the convection field in the midpoint of $I_n$, which restores a
  Kronecker product structure inside the smoother without altering
  the outer FGMRES operator; cf.~\Cref{subsec:VankaSmth}.
\end{remark}

\section{\label{sec:AlgSys}Algebraic system}

Here, we derive the algebraic form of Problem~\ref{Prob:DVP} by exploiting the space-time tensor product structure~\eqref{eq:GDS} of the discrete spaces. The tensor product form is preserved in the algebraic system. This simplifies the assembly of the finite element matrices by factorizing into space and time integrals. The  matrix-free framework becomes efficient. The treatment of the nonlinear convective term is more involved, but it can still be captured in this tensor product approach. An embedding of the Newton linearized, tensor product structured algebraic system into an \(hp\) multigrid preconditioning concept is developed in \Cref{sec:mg-framework}.

\subsection{Preliminaries}\label{Sec:Prem}
For the evaluation of the time integrals in~\eqref{eq:DNSE}, we employ the right-sided $(k+1)$-point
Gauss-Radau quadrature on $I_n=(t_{n-1},t_n]$,
\begin{equation}
  \label{eq:GF}
  Q_n(w) \coloneq \frac{\tau_n}{2}\sum_{\mu=1}^{k+1} \hat\omega_\mu\, w(t_n^{\mu}) \approx \int_{I_n} w(t)\d t\,,
\end{equation}
where $t_n^{\mu}=T_n(\hat t_\mu)$ with $T_n(\hat t)\coloneq(t_{n-1}+t_n)/2+ (\tau_n/2)\hat t$ and
$\{\hat t_\mu,\hat\omega_\mu\}_{\mu=1}^{k+1}$ are the Gauss--Radau points and corresponding weights on $[-1,1]$.
The rule~\eqref{eq:GF} is exact for all $w\in\mathbb P_{2k}(I_n;\mathbb R)$ and
$t_n^{k+1}=t_n$. Therefore, it is not exact for the nonlinear terms in \eqref{eq:DNSE}. The quadrature error will be analyzed below in Lemma.~\ref{lem:GR-div}.

For the temporal finite element space~\eqref{eq:DefYtau}, we use the local Lagrange basis associated with the
Gauss--Radau nodes and supported on single subintervals $I_n$,
\begin{equation}
  \label{eq:BasYk}
  \begin{aligned}
    Y_\tau^k(I) = \operatorname{span}\big\{& \varphi_{n}^a \in L^2(I) \mid (\varphi^a_{n})_{|I_b}\in \mathbb P_k(I_b;\mathbb R),\ b=1,\ldots,N\,,\; \operatorname{supp}\, \varphi^a_{n}\subset \overline I_n\,, \\
    & \varphi^a_{n} (t_{n}^{\mu}) = \delta_{a,\mu},\ \mu =1,\ldots, k+1\,, \; a=1,\ldots,k+1,\ n=1,\ldots, N \big \}\,,
  \end{aligned}
\end{equation}
with the Kronecker symbol $\delta_{a,\mu}$. For the spatial discretization, we introduce the global
finite element bases associated with the finite element spaces~\eqref{Def:VhQh} by
\begin{equation}
  \label{eq:BasVhQh}
    \boldsymbol V_h^{r+1}(\Omega) = \operatorname{span}\{\boldsymbol\chi^{\boldsymbol v}_m\}_{m=1}^{M^{\boldsymbol v}}\,, \qquad
    Q_h^{r}(\Omega) = \operatorname{span}\{\chi^{p}_m\}_{m=1}^{M^p}\,.
\end{equation}
On each slab $\Omega \times I_n$, the tuple $(\boldsymbol v_{\tau h},p_{\tau h})\in \boldsymbol H_{\tau h}^{\boldsymbol v}\times H_{\tau h}^p$ has the tensor-product form
\begin{equation}
  \label{eq:Repvp}
    \boldsymbol v_{\tau h}{}_{|I_n} =  \sum_{a=1}^{k+1}\sum_{m=1}^{M^{\boldsymbol v}} v_n^{a,m}\,\varphi_n^a\,\boldsymbol\chi_m^{\boldsymbol v}\,,	\quad p_{\tau h}{}_{|I_n} = \sum_{a=1}^{k+1}\sum_{m=1}^{M^{p}} p_n^{a,m}\,\varphi_n^a\,\chi_m^{p}
\end{equation}
with coefficients $v_n^{a,m}\in\mathbb R^d$ and $p_n^{a,m}\in\mathbb R$, for $a=1,\ldots,k+1$ and $m=1,\ldots,M$, with $M\in \{M^{\boldsymbol v},M^p\}$. On each slab, we assemble the vectors of unknowns by
\begin{subequations}
\label{eq:DefSubvec_0}
\begin{alignat}{4}
\label{eq:DefSubvec_1}
  \boldsymbol V_n^{a} & \coloneq (v_n^{a,1},\ldots,v_n^{a,M^{\boldsymbol v}})^{\top}\in\mathbb R^{M^{\boldsymbol v}} \,,
  & \boldsymbol P_n^{a} & \coloneq (p_n^{a,1},\ldots,p_n^{a,M^p})^{\top}\in\mathbb R^{M^p}\,, \\
\label{eq:DefSubvec_2}
  \boldsymbol V_{n}  & \coloneq\big(\boldsymbol V_{n}^{1},\ldots ,\boldsymbol V_{n}^{k+1}\big)^{\top}\in \mathbb R^{(k+1)M^{\boldsymbol v}}\,,
  & \quad \boldsymbol P_{n}  & \coloneq\big(\boldsymbol P_{n}^{1},\ldots,\boldsymbol P_{n}^{k+1}\big)^{\top}\in \mathbb R^{(k+1)M^p}\,.
\end{alignat}
\end{subequations}

\paragraph{Assembly of the bilinear and linear terms in~\eqref{eq:DNSE}}

The matrices assembled from the bilinear contributions in~\eqref{eq:DNSE} for the tensor product spaces~\eqref{eq:GDS} and their bases in~\eqref{eq:BasYk} and~\eqref{eq:BasVhQh}, respectively, are defined explicitly in Appendix~\ref{App:Assembly}. The algebraic counterpart of the linear forms in~\eqref{eq:DNSE} are summarized in Appendix~\ref{App:Assembly} as well.

\paragraph{Assembly of the nonlinear convective terms in~\eqref{eq:DNSE}}
For the tensor product framework, the treatment of the nonlinear convective terms in~\eqref{eq:DNSE} contributing to $A_\gamma^{tp}$ and $B_\gamma$ is more involved. For the application of Newton's method to the nonlinear system, the Jacobian of the nonlinear contributions is needed further. This is derived in the sequel. We let $c:\boldsymbol V_h^{r+1}\times \boldsymbol V_h^{r+1} \to \R$ be defined by
\begin{equation}
\label{eq:Defc}
c(\boldsymbol v_h)(\boldsymbol z_h) \coloneqq - \langle \boldsymbol v_h \otimes \boldsymbol v_h , \nabla \boldsymbol z_h \rangle + \langle (\boldsymbol v_h  \cdot \boldsymbol n) \boldsymbol v_h, \boldsymbol z_h \rangle_\Gamma
\end{equation}
for $\boldsymbol v_h, \boldsymbol z_h\in \boldsymbol V_h^{r+1}$. Its Gateaux derivative at $\boldsymbol v_h$ in the direction $\boldsymbol{\hat v_h}\in \boldsymbol V_h^{r+1}$ is
\begin{equation}
\label{eq:Defcp}
\begin{aligned}
c'(\boldsymbol v_h)(\boldsymbol{\hat v_h},\boldsymbol z_h) & = - \langle \boldsymbol{\hat v_h} \otimes \boldsymbol v_h , \nabla \boldsymbol z_h \rangle - \langle  \boldsymbol v_h \otimes \boldsymbol{\hat v_h} , \nabla \boldsymbol z_h \rangle  \\
& \quad + \langle (\boldsymbol{\hat v}_h  \cdot \boldsymbol n) \boldsymbol v_h, \boldsymbol z_h \rangle_\Gamma + \langle (\boldsymbol{v}_h  \cdot \boldsymbol n) \boldsymbol{\hat v}_h, \boldsymbol z_h \rangle_\Gamma\,, \quad \forall \boldsymbol z_h\in \boldsymbol V_h^{r+1}\,.
\end{aligned}
\end{equation}

For some $\boldsymbol v_h \in \boldsymbol V_h^{r+1}$, let $\boldsymbol V\in\mathbb R^{M^{\boldsymbol v}}$ denote the coefficient vector with respect to the representation of $\boldsymbol v_h$ in the basis $\{\boldsymbol \chi_i^{\boldsymbol v}\}_{m=1}^{}$ of $\boldsymbol V_h^{r+1}$; cf.~\eqref{eq:BasVhQh}. By means of~\eqref{eq:Defc}, we let $\boldsymbol H: \R^{M^{\boldsymbol v}} \to \R^{M^{\boldsymbol v}}$, $\boldsymbol V \to H(\boldsymbol V)$, with $\boldsymbol v_h = \sum_{m=1}^{M^{\boldsymbol v}} V_m \boldsymbol \chi_m^{\boldsymbol v}\in \boldsymbol V_h^{r+1}$, be given by
\begin{equation}
\label{eq:DefH}
(\boldsymbol H(\boldsymbol V))_m \coloneqq  c(\boldsymbol v_h)(\boldsymbol \chi_m^{\boldsymbol v})
\end{equation}
for $m =1,\ldots, M^{\boldsymbol v}$. For some time slab vector $\boldsymbol V_n=(\boldsymbol V_n^1,\ldots , \boldsymbol V_n^{k+1})^\top \in \R^{(k+1)\cdot M^{\boldsymbol{v}}}$, we let $\boldsymbol H: \R^{(k+1)\cdot M^{\boldsymbol v}} \to \R^{(k+1)\cdot M^{\boldsymbol v}}$, $\boldsymbol V \to H(\boldsymbol V)$ be given by
\begin{equation}
\label{eq:DefHslab}
\boldsymbol H_n(\boldsymbol V_n)\coloneqq (\boldsymbol H(\boldsymbol V_n^1),\ldots, \boldsymbol H(\boldsymbol V_n^{k+1}))^\top\,.
\end{equation}

Using~\eqref{eq:GF} for the basis of~\eqref{eq:BasYk}, it follows that
\begin{equation}
\label{eq:QuadNL}
\sum_{b,c=1}^{k+1} \int_{I_n} \varphi_n^b(t) \varphi_n^c(t) \varphi_n^a(t) \d t \approx  \frac{\tau_n}{2}\sum_{\mu=1}^{k+1} \hat \omega_\mu \varphi_n^b(t_n^\mu) \varphi_n^c(t_n^\mu) \varphi_n^a(t_n^\mu)\,,
\end{equation}
for $a,b,c\in \{1,\ldots, k+1\}$. We note that~\eqref{eq:QuadNL} does not amount to an identity since~\eqref{eq:BasYk} is exact for all polynomials of order less or equal than $2k$ only, and $\varphi_n^b(t) \varphi_n^c(t) \varphi_n^a \in \mathbb P_{3k}(I_n)$. The quadrature error in~\eqref{eq:QuadNL} is analyzed below in \Cref{sec:gr-nonlinear}. From~\eqref{eq:QuadNL}, we conclude that, with a quadrature error $\boldsymbol{E}_\text{c}^{\text{GR}}$,
\begin{equation}
\label{eq:AlgFconv}
\left(\left(\int_{I_n} c(\boldsymbol v_{\tau,h})(\varphi_n^a \boldsymbol \chi_m^{\boldsymbol v})\, \mathrm d t\right)_{m=1}^{M^{\boldsymbol v}}\right)_{a=1}^{k+1} = (\boldsymbol M_n^\tau \otimes \boldsymbol I) \boldsymbol H_n(\boldsymbol V_n) + \boldsymbol{E}^\text{GR}_{\text{c}}\,,
\end{equation}
where $\boldsymbol M_n^\tau\in \R^{k+1,k+1}$ is diagonal and defined in~\eqref{eq:DefKM}. The tensor (or right Kronecker) product  $\boldsymbol A\otimes \boldsymbol B$ of $\boldsymbol A\in \R^{r,r}$ and $\boldsymbol B\in \R^{s,s}$, for $r,s\in \N$, is given by
\begin{equation}
  \label{Eq:DefKr}
  \boldsymbol A \otimes \boldsymbol B \coloneqq \left(a_{ij}\boldsymbol B\right)_{i,j=1}^{r}\,.
\end{equation}

\begin{remark}
\label{Rem:AssNConvBd}
The contributions to~\eqref{eq:DNSE} that arise from the boundary pairings with the convective term in \eqref{eq:def-b-gam-2} and \eqref{eq:def-a-gam-tp} are assembled similarly to~\eqref{eq:AlgFconv} as $(\boldsymbol M_n^\tau \otimes \boldsymbol I) \boldsymbol N_{\Gamma_D}^{\boldsymbol v,c}(\boldsymbol V_n)$.
\end{remark}

For solving the nonlinear algebraic system we apply Newtons's method. For this, the Jacobian of the nonlinear mapping $\boldsymbol H_n$ in~\eqref{eq:DefHslab} is needed. Recalling~\eqref{eq:Defcp} or, alternatively, building the derivative of the mapping $\boldsymbol H$ in~\eqref{eq:DefH} with respect to the unknown vector $\boldsymbol V$, we get the mapping $\boldsymbol H': \R^{M^{\boldsymbol v}} \to \R^{M^{\boldsymbol v},M^{\boldsymbol v}}$, $\boldsymbol V \mapsto \boldsymbol{H}'(\boldsymbol V)$ with
\begin{equation}
\label{eq:DefHp}
(\boldsymbol{H}'(\boldsymbol V))_{ms} = \frac{\partial H_m}{\partial V_s} = c'(\boldsymbol v_h)(\boldsymbol \chi_s^{\boldsymbol v},\boldsymbol \chi_m^{\boldsymbol v})
\end{equation}
for $m,s=1,\ldots, M^{\boldsymbol v}$. The Jacobian of $\boldsymbol H_n$ in~\eqref{eq:DefHslab} is then computed as $\boldsymbol H_n': \R^{(k+1)\cdot M^{\boldsymbol v}} \to \R^{(k+1)\cdot M^{\boldsymbol v},(k+1)\cdot M^{\boldsymbol v}}$, $\boldsymbol V_n \mapsto \boldsymbol{H}_n'(\boldsymbol V)$ with the block diagonal structure
\begin{equation}
\label{eq:DefHnp}
(\boldsymbol{H}_n'(\boldsymbol V_n)) = \operatorname{diag}(\boldsymbol{H}'(\boldsymbol V_n^1),\ldots ,\boldsymbol{H}'(\boldsymbol V_n^{k+1}))\,.
\end{equation}

\subsection{Temporal quadrature approximation of the nonlinear term}
\label{sec:gr-nonlinear}

Here we analyze briefly the consistency error $\boldsymbol E_c^{\text{GR}}$ in~\eqref{eq:AlgFconv}. It occurs if the convective term in~\eqref{eq:DNSE}, that is a polynomial of order $3k$ in time, is integrated in time by the Gauss--Radau quadrature formula~\eqref{eq:GF}, that is exact for all polynomials of order less or equal then $2k$ only. We present an error estimate for the time integration of the convective term that holds under suitable stability assumptions about the fully discrete solution. Rigorous error estimates for the overall scheme are beyond the scope of interest in this paper. The convergence behavior of our scheme is illustrated numerically in \Cref{sec:experiments}. This further illustrates that the approach is methodologically sound.

\begin{lemma}[Gauss--Radau quadrature error for divergence-form of convection]\label{lem:GR-div}
  For $k\in \N_0$, suppose that $\boldsymbol u_{\tau h} \in Y_\tau^k(I) \otimes \boldsymbol V_h^{\operatorname{div}}$ satisfies, for $l=0,\ldots,k$,
  \begin{equation}
  \label{eq:Reguwth}
  \max_{n=1,\ldots,N} \big\{ \| \partial_t^l \boldsymbol u_{\tau h} \|_{L^\infty(I_n;\boldsymbol{H}^1(\Omega))} \| + \| \widetilde \Delta_h  \partial_t^l \boldsymbol u_{\tau h} \|_{L^\infty(I_n;\boldsymbol{L}^2(\Omega))}\big\} \leq A_k\,,
  \end{equation}
where $\widetilde \Delta_h$ denotes the discrete Stokes operator (cf.~\cite[p.~297]{HeywoodRannacher1982}), and $A_k$ is independent of the mesh sizes $\tau_n$ and $h$. For $\boldsymbol w_{\tau h} \in Y_\tau^k(I) \otimes \boldsymbol V_h^{r+1}$ let
  \begin{equation*}
    \mathcal C_n(\boldsymbol u_{\tau h})(\boldsymbol w_{\tau h}) \coloneq \int_{I_n} c(\boldsymbol u_{\tau h}(t))(\boldsymbol w_{\tau h}(t))\, \mathrm dt\,.
  \end{equation*}
  For the $(k+1)$-point Gauss--Radau formula~\eqref{eq:GF}, we put
  \begin{equation*}
    \mathcal C^{\operatorname{GR}}_n(\boldsymbol u_{\tau h})(\boldsymbol w_{\tau h}) \coloneq \sum_{\mu=1}^{k+1} w_{n,\mu}\, c(\boldsymbol u_{\tau h}(t_n^\mu))(\boldsymbol w_{\tau h}(t_n^\mu))\,.
  \end{equation*}
  There exists a constant $B_k=B_k(A_k)$ such that there holds that
  \begin{equation}
  \label{eq:QErr0}
  \bigg| \sum_{n=1}^N \bigg(\mathcal C_n(\boldsymbol u_{\tau h})(\boldsymbol w_{\tau h})-\mathcal C_n^{\mathrm{GR}}(\boldsymbol u_{\tau h})(\boldsymbol w_{\tau h})\bigg) \bigg| \leq B_k \tau^{2k+1} \|\boldsymbol w_{\tau h} \|_{L^2(I;\boldsymbol H^1(\Omega))} \,.
  \end{equation}
\end{lemma}

\begin{proof}
For given $\boldsymbol u_{\tau h} \in Y_\tau^k(I) \otimes \boldsymbol V_h^{\operatorname{div}}$ and $\boldsymbol w_{\tau h} \in Y_\tau^k(I) \otimes \boldsymbol V_h^{r+1}$, we have that $c(\boldsymbol u_{\tau h})(\boldsymbol w_{\tau h})\in \mathbb P_{3k}(I_n;\R)$. Since the Gauss--Radau quadrature formula is exact for polynomials in $\mathbb P_{2k}(I_n;\R)$, the Peano kernel remainder yields that
\begin{align}
  \nonumber
  E_n^{\text{GR}}& \coloneqq \mathcal C^{\operatorname{GR}}_n(\boldsymbol u_{\tau h})(\boldsymbol w_{\tau h}) - \mathcal C^{\operatorname{GR}}_n(\boldsymbol u_{\tau h})(\boldsymbol w_{\tau h})\\
  &
  \nonumber
   = \kappa_k\,\tau_n^{2k+2}\,\partial_t^{2k+2} c(\boldsymbol u_{\tau h}(\xi_n))(\boldsymbol w_{\tau h}(\xi_n))  \\
\label{eq:QErr1}
  & = \kappa_k\,\tau_n^{2k+2} \sum_{|\alpha | \leq 2k+2\atop \alpha_i\leq k, \; i\in \{1,2,3\}} c\big(\partial_t^{(\alpha_1,\alpha_2)}  \boldsymbol u_{\tau h}(\xi_n)\big)\big(\partial_t^{\alpha_3} \boldsymbol w_{\tau h}(\xi_n)\big)
\end{align}
for some $\xi_n\in I_n$, with $\kappa_k$ depending only on $k$. Here, the multiindex $(\alpha_1,\alpha_2)$ denotes the respective order of the time derive in the product of $\boldsymbol u_{\tau h}$ with itself in~\eqref{eq:Defc}, i.e., $\partial_t^{(\alpha_1,\alpha_2)} (\boldsymbol u_{\tau h} \otimes \boldsymbol u_{\tau h}) \coloneqq \partial_t^{\alpha_1} \boldsymbol u_{\tau h} \otimes \partial_t^{\alpha_2}  \boldsymbol u_{\tau h}$. We recall that $\boldsymbol u_{\tau h},\boldsymbol w_{\tau h}\in Y_\tau^k(I) \otimes \boldsymbol V_h^{r+1}$, such that their time derivatives of order $k+1$ or higher vanish.

Now, let $\boldsymbol u_h \in \boldsymbol V_h^{\operatorname{div}}$, $\boldsymbol w_h \in \boldsymbol V_h^{r+1}$ be given. Recalling~\eqref{eq:Defc}, there holds by the H\"older inequality and embedding theorems for the volume integrals and by \cite[Eq.\ (4.46)]{HeywoodRannacher1982} for the boundary terms that
\begin{align}
\label{eq:QErr2}
|c(\boldsymbol u_h)(\boldsymbol w_h)| & \leq c \big(\|\nabla \boldsymbol u_h\|^2 \|\nabla \boldsymbol w_h\| +   \| \widetilde \Delta_h \boldsymbol u_h\|^{1/2}\|\nabla \boldsymbol u_h\|^{1/2} (\|\boldsymbol w_h\| + h \|\nabla \boldsymbol w_h\|)\big)\,.
\end{align}
Summing up~\eqref{eq:QErr1} from $n=1$ to $N$ and using~\eqref{eq:QErr2} along with~\eqref{eq:Reguwth}  yields that
\begin{equation}
\label{eq:QErr3}
\bigg| \sum_{n=1}^N E_n^{\text{GR}} \bigg|  \leq B_k \sum_{n=1}^N \tau_n^{2k+2}  \sum_{l=0}^k \| \partial_t^l  \boldsymbol w_{\tau h}\|_{L^\infty(I_n;\boldsymbol H^1(\Omega))}\,.
\end{equation}
By the $L^\infty$--$L^2$ inverse property \cite[Eq.\ (2.5)]{KarakashianMakridakis2005}
\begin{equation*}
\| y \| _{L^{\infty}(I_n;\R)}\leq c \tau_n^{-1/2} \| y \|_{L^2(I_n;\R)} \,, \quad \text{for }\;  y \in \mathbb P_k(I_n;\R)\,,
\end{equation*}
and the $H^1$--$L^2$ inverse property
\begin{equation*}
  \| \partial_t y \| _{L^{2}(I_n;\R)}\leq c \tau_n^{-1} \| y \|_{L^2(I_n;\R)} \,, \quad \text{for }\;  y \in \mathbb P_k(I_n;\R)\,,
\end{equation*}
we deduce from~\eqref{eq:QErr3} that
\begin{equation}
  \label{eq:QErr4}
  \bigg| \sum_{n=1}^N E_n^{\text{GR}} \bigg|  \leq B_k \sum_{n=1}^N   \tau_n^{k+3/2}  \|  \boldsymbol w_{\tau h}\|_{L^2(I_n;\boldsymbol H^1(\Omega))}.
\end{equation}
From~\eqref{eq:QErr4}, we conclude assertion~\eqref{eq:QErr0} by the inequality of Cauchy--Schwarz.
\end{proof}

\begin{remark}[On the result of Lemma~\ref{lem:GR-div}]
\begin{itemize}[leftmargin=*,itemsep=0pt,topsep=0pt,parsep=0pt]
\item The stability bounds in~\eqref{eq:Reguwth} are not straightforward. For the regularity of continuous and discrete solutions to the Navier--Stokes equations, the occurrence of non-local compatibility conditions, and uniform bounds up to $t=0$, we refer to the comprehensive literature, in particular~\cite{HeywoodRannacher1982,HeywoodRannacher1990}.

\item  Inequality~\eqref{eq:QErr0} yields an error estimation of the form as it is required and often applied in error analyses for Navier--Stokes approximations. After the additional application of the inequalities of Cauchy--Schwarz and Cauchy--Young  and a suitable choice of the test function $\boldsymbol w_{\tau h} $, the term $ \|\boldsymbol w_{\tau h} \|_{L^2(I;\boldsymbol H^1(\Omega))}$ can be absorbed by the viscous term of the error identity to the considered scheme; cp., e.g., \cite{john_FiniteElement_2016}.
\end{itemize}
\end{remark}

\subsection{Algebraic form of the discrete problem}
\label{Sec:Afdp}

Now, we rewrite Problem~\ref{Prob:DVP} in its algebraic form. By the choice of a local temporal basis in~\eqref{eq:BasYk}, supported on the subintervals $I_n$, we end up with a time marching scheme. In each time step, the nonlinear system of equations is solved by an inexact Newton--Krylov method using FGMRES iterations with $hp$ multigrid preconditioning. Even if three space dimensions and higher order approximations are involved, the time marching approach is economical and becomes still feasible without tremendous computing power and memory resources. In contrast to this, a holistic approach solves the global in time nonlinear algebraic system; cf.\ Remark \ref{rem:HolSys}. In the Newton iteration, the block lower bi-diagonal structure of the global algebraic system can then be used for building a time marching process and splitting the global system into a sequence of local algebraic problems again. This approach is not studied here. We recast Problem~\ref{Prob:DVP}, up to the quadrature error in the convective terms (cf.~\Cref{sec:gr-nonlinear}), in the following form.

\begin{problem}[Local algebraic Navier--Stokes problem]\label{Prob:LocAlgNS}
  Let $n\in\{1,\ldots,N\}$. For $n>1$ set $\boldsymbol v_{\tau h}(t_{n-1})=\sum_{m=1}^{M^{\boldsymbol v}} v_{n-1}^{k+1,m}\,\boldsymbol\chi_m^{\boldsymbol v}$, and for $n=1$ set $\boldsymbol v_{\tau h}(t_0)=\boldsymbol v_{0,h}=\sum_{m=1}^{M^{\boldsymbol v}} v_0^{m}\,\boldsymbol\chi_m^{\boldsymbol v}$. Define $\boldsymbol V_{n-1}\in \R^{(k+1)\cdot M^{\boldsymbol v}}$ as
  \begin{equation}
    \label{eq:DefVnm1}
    \boldsymbol{V}_{n-1}\coloneq
    \left\{\begin{array}{@{}ll}
      \big(\boldsymbol 0, \ldots, \boldsymbol 0,v_{n-1}^{k+1,1},\ldots,v_{n-1}^{k+1,M^{\boldsymbol v}} \big)^\top\,, & \text{for } n>1\,,\\[3pt]
      \big(\boldsymbol 0, \ldots, \boldsymbol 0,v_{0}^{1},\ldots,v_{0}^{M^{\boldsymbol v}} \big)^\top\,, & \text{for } n=1\,.
    \end{array} \right.
  \end{equation}
Find $\boldsymbol U_n \coloneqq (\boldsymbol V_n,\boldsymbol P_n)\in\mathbb R^{(k+1)\cdot (M^{\boldsymbol v}+M^p)}$ such that
  \begin{subequations}
    \label{eq:InAlgNS}
    \begin{alignat}{2}
      \nonumber
      & \underbrace{(\boldsymbol K_n^\tau\otimes \boldsymbol M_h)\,\boldsymbol V_n}_{\text{time derivative}}
       + \underbrace{(\boldsymbol M_n^\tau\otimes \boldsymbol I)\,\boldsymbol H_n(\boldsymbol V_n)}_{\text{convection (divergence form)}} + \underbrace{ (\boldsymbol M_n^\tau\otimes \nu\,\boldsymbol A_h)\,\boldsymbol V_n}_{\text{viscous}}  \\
      \nonumber
        & \quad + \underbrace{(\boldsymbol M_n^\tau\otimes \boldsymbol B_h^{\top})\,\boldsymbol P_n}_{\text{pressure coupling}} + \underbrace{(\boldsymbol M_n^\tau \otimes \boldsymbol I) \boldsymbol N_{\Gamma_D}^{\boldsymbol v,c}(\boldsymbol V_n)}_{\text{Boundary pairings with} \atop \text{convective term in~\eqref{eq:def-a-gam-tp}}} + \underbrace{(\boldsymbol M_n^\tau\otimes \boldsymbol N^{\boldsymbol v;b,r}_{\Gamma_D}(\gamma))\,\boldsymbol V_n}_{\text{boundary condition and} \atop \text{Nitsche's term $B^s+B_\gamma^r$ in~\eqref{eq:def-a-gam-tp}}} \\
            \label{eq:InAlgNS1}
      &\quad + \underbrace{(\boldsymbol M_n^\tau\otimes (\boldsymbol G^{p}_{\Gamma_D})^\top)\boldsymbol P_n}_{\text{boundary condition}} = \underbrace{\boldsymbol F_n}_{\text{body force}}  + \underbrace{\boldsymbol L_n}_{\text{Dirichlet data}\atop \text{via Nitsche}}  + \underbrace{(\boldsymbol C_n^\tau\otimes \boldsymbol M_h)\,\boldsymbol V_{n-1}}_{\text{jump at $t_{n-1}$ of DG in time}}\,,\\[1ex]
    \label{eq:InAlgNS2}
        & \underbrace{(\boldsymbol M_n^\tau\otimes \boldsymbol B_h)\,\boldsymbol V_n}_{\text{continuity}}  + \underbrace{(\boldsymbol M_n^\tau\otimes \boldsymbol G^{p}_{\Gamma_D})\,\boldsymbol V_n }_{\text{Nitsche's term $B^s$ in~\eqref{eq:def-a-gam-tp}}} = \boldsymbol 0 \,.
    \end{alignat}
  \end{subequations}
\end{problem}

On $I_n$, $\boldsymbol u_{\tau h} \in \boldsymbol{X}_{\tau h}^{kr}$ is then defined by the expansions in~\eqref{eq:Repvp} along with~\eqref{eq:DefSubvec_0}. For the definition of the quantities in~\eqref{eq:InAlgNS} we refer to \Cref{Sec:Prem} and Appendix~\ref{App:Assembly}. Problem~\ref{Prob:LocAlgNS} leads to a global in time system with block lower bi-diagonal structure.

\begin{definition}[Global algebraic Navier--Stokes problem]
\label{rem:HolSys}
Suppose that $\boldsymbol{V}_0$ is given by \eqref{eq:DefVnm1}. Find $\boldsymbol{U} \coloneqq (\boldsymbol{U}_1,\ldots ,\boldsymbol{U}_n)\in \R^{N\cdot (k+1)\cdot (M^{\boldsymbol v}+M^p)}$, such that 
\begin{equation}\label{eq:GloNS}
  \boldsymbol{\mathcal R}(\boldsymbol U) = \boldsymbol 0\,, \qquad \text{with }\; \boldsymbol{\mathcal R}(\boldsymbol U) \equiv
  \begin{pmatrix}
    \boldsymbol{\mathcal R}_1(\boldsymbol U_1;\boldsymbol V_0)\\
    \boldsymbol{\mathcal R}_2(\boldsymbol U_2;\boldsymbol V_1)\\
    \vdots \\
    \boldsymbol{\mathcal R}_N(\boldsymbol U_{N};\boldsymbol V_{N-1})
  \end{pmatrix}\,,
\end{equation}
with $\boldsymbol{\mathcal R}_n: \R^{(k+1)\cdot (M^{\boldsymbol v}+M^p)}\to \R^{(k+1)\cdot (M^{\boldsymbol v}+M^p)}$, for $n=1,\ldots,N$, being defined by~\eqref{eq:InAlgNS}.
\end{definition}

To solve the sequence of local in time problems~\eqref{eq:InAlgNS}, which amounts to solving~\eqref{eq:GloNS} row-wise, we need the Jacobian matrix of the mapping in~\eqref{eq:InAlgNS}. It is given by
\begin{equation}\label{eq:JacobianLocal0}
\boldsymbol{\mathcal J}_n(\boldsymbol U_n) = \begin{pmatrix} \boldsymbol{\mathcal J}_n^{1,1}(\boldsymbol U_n)_{1,1} & \boldsymbol{\mathcal J}_n^{1,2}(\boldsymbol U_n)\\[1ex] \boldsymbol{\mathcal J}_n^{2,1}(\boldsymbol U_n)_{2,1} & \boldsymbol{\mathcal J}_n^{2,2}(\boldsymbol U_n) \end{pmatrix} \in \R^{(k+1)\cdot (M^{\boldsymbol v}+M^p),(k+1)\cdot (M^{\boldsymbol v}+M^p)}\,,
\end{equation}
with submatrices $\boldsymbol{\mathcal J}_n^{1,1}\in \R^{(k+1)\cdot M^{\boldsymbol v},(k+1)\cdot M^{\boldsymbol v}}$, $\boldsymbol{\mathcal J}_n^{1,2}\in \R^{(k+1)\cdot M^{\boldsymbol v},(k+1)\cdot M^p}$ and $\boldsymbol{\mathcal J}_n^{2,2}\in \R^{(k+1)\cdot M^{p},(k+1)\cdot M^{p}}$ being defined by
\begin{subequations}
\label{eq:JacobianLocal1}
\begin{alignat}{4}
\label{eq:JacobianLocal2}
\boldsymbol{\mathcal J}_n^{1,1}(\boldsymbol U_n) & = \boldsymbol K_n^\tau\otimes \boldsymbol M_h + (\boldsymbol M_n^\tau\otimes \boldsymbol I)\,\boldsymbol H_n' (\boldsymbol V_n) + \boldsymbol M_n^\tau\otimes \,\nu \, \boldsymbol A_h\\
\nonumber& \quad  + (\boldsymbol M_n^\tau \otimes \boldsymbol I) (\boldsymbol N_{\Gamma_D}^{\boldsymbol v,c})' (\boldsymbol V_n)+ \boldsymbol M_n^\tau\otimes \boldsymbol N^{\boldsymbol v;b,r}_{\Gamma_D}(\gamma)\,,\\[1ex]
\label{eq:JacobianLocal3}
\boldsymbol{\mathcal J}_n^{1,2}(\boldsymbol U_n) & = \boldsymbol M_n^\tau\otimes  (\boldsymbol B_h+ \boldsymbol G^{p}_{\Gamma_D})\,, \quad \boldsymbol{\mathcal J}_n^{2,1}(\boldsymbol U_n) = \boldsymbol{\mathcal J}_n^{1,2}(\boldsymbol U_n)^\top \,,\\
\boldsymbol{\mathcal J}_n^{2,2}(\boldsymbol U_n) & = \boldsymbol{0}\,.
\end{alignat}
\end{subequations}
For this, we recall~\eqref{eq:DefHnp}. We note that $(\boldsymbol A \otimes \boldsymbol{B})^\top = \boldsymbol A^\top \otimes \boldsymbol B^\top $ is satisfied. The matrix $\boldsymbol M_n^\tau$ is symmetric by its definition. By~\eqref{eq:JacobianLocal3}, the Jacobian matrix has a saddle point structure. Moreover, a tensor product structure is preserved for its submatrices.

\section{Solution of the nonlinear system of equations}
\label{sec:nonlinear-solver}

We solve~\eqref{eq:InAlgNS}, recast in the system~\eqref{eq:GloNS}, by a Newton--Krylov method with FGMRES iterations and \(hp\) space-time multigrid preconditioning in a matrix-free framework. This is presented now in detail. The algorithms are summarized in Appendix~\ref{App:Algorithms}.

\subsection{Inexact Newton--Krylov with Armijo globalization}
\label{sec:ink}
We solve~\eqref{eq:GloNS} row-wise. For this, we use Newton's method with
globalization by a nonmonotone Armijo rule~\cite{GrippoLamparielloLucidi1986}.
We recall it briefly for completeness. Starting with some initial guess, suppose that the iterate $\boldsymbol{U}_n^{m}\in \R^{(k+1)\cdot (M^{\boldsymbol v}+M^p)}$ has been computed. Then, we calculate the Newton correction $\boldsymbol{\widehat U}^{m}_n\in \R^{(k+1)\cdot (M^{\boldsymbol v}+M^p)}$ by solving with the FGMRES method \cite{Saad1993} the system
\begin{equation}\label{eq:lin-sys}
   \boldsymbol{\mathcal J}_n(\boldsymbol U^{m}_n)\,\boldsymbol{\widehat U}^{m}_n= -\,\boldsymbol{\mathcal R}_n(\boldsymbol U_n^{m};\boldsymbol V_{n-1})\,,
\end{equation}
with the stopping criterion that
\begin{equation}\label{eq:inexact}
  \| \boldsymbol{\mathcal J}_n(\boldsymbol U^{m}_n)\,\boldsymbol{\widehat U}^{m}_n + \boldsymbol{\mathcal R}_n(\boldsymbol U_n^{m};\boldsymbol V_{n-1})\ \big\|_{\boldsymbol{\mathcal M}}
  \le \eta_m \,\big\| \boldsymbol{\mathcal R}_n(\boldsymbol U_n^{m};\boldsymbol V_{n-1})\big\|_{\boldsymbol{\mathcal M}}\,,
\end{equation}
where the mass-weighted norm is defined by
\[
  \boldsymbol{\mathcal M}_n \coloneqq
  \begin{pmatrix}
    \boldsymbol M_n^\tau\otimes \boldsymbol M_h & 0\\
    0 & \boldsymbol M_n^\tau\otimes \boldsymbol M_h^{p}
  \end{pmatrix},
  \qquad
  \|\boldsymbol Z\|_{\boldsymbol{\mathcal M}_n} \coloneqq \big(\boldsymbol Z^\top \boldsymbol{\mathcal M}_n\,\boldsymbol Z\big)^{1/2},
\]
for $\boldsymbol Z=(\boldsymbol Z^{\boldsymbol v},\boldsymbol Z^{p})$.
For some $\eta_{\max},\, c_{\mathrm{EW}}\in(0,1)$ and $\gamma_{\mathrm{EW}}\in(1,2]$, the parameter $\eta_m$ is chosen as in~\cite{EisenstatWalker1996}
\begin{equation}\label{eq:ew}
  \eta_m = \min\left\{\eta_{\max},\,
  c_{\mathrm{EW}}
  \left(\frac{\|\boldsymbol{\mathcal R}_n(\boldsymbol U_n^{m};\boldsymbol V_{n-1})\|_{\boldsymbol{\mathcal M}_n}}
             {\max\{\|\boldsymbol{\mathcal R}_n(\boldsymbol U_n^{m-1};\boldsymbol V_{n-1})\|_{\boldsymbol{\mathcal M}_n},\,\varepsilon\}}
  \right)^{\gamma_{\mathrm{EW}}}\right\}\,.
\end{equation}
Our globalization approach uses Armijo's backtracking with the merit function
\[\phi(\alpha)\coloneq\tfrac12\|\boldsymbol{\mathcal R}_n(\boldsymbol U_n^{m}+\alpha\,\boldsymbol{\widehat U}_n^{m};\boldsymbol V_{n-1})\|_{\boldsymbol{\mathcal M}}^{2}\,.\]
For $g_0 \coloneqq \phi'(0)= \boldsymbol{\mathcal R}_n(\boldsymbol U_n^{m};\boldsymbol V_{n-1})^\top \boldsymbol{\mathcal J}_n(\boldsymbol U^{m}_n)\boldsymbol{\widehat U}^{m}_n$ and the largest $\alpha\in(0,1]$ such that
\begin{equation}\label{eq:armijo}
  \phi(\alpha) \le \phi(0) + c_1\,\alpha\,g_0, \qquad \text{for }\, c_1\in(0,1)\,,
\end{equation}
is satisfied, we then put
\begin{equation}
\boldsymbol U_n^{m+1} = \boldsymbol U_n^{m+1} + \alpha \boldsymbol{\widehat U}_n^{m}\,.
\end{equation}
The Newton iteration along with its control is summarized in Algorithm~\ref{alg:ink} to~\ref{alg:armijo}.

\subsection{\label{sec:mg-framework} Preconditioning of GMRES by \(hp\) space-time multigrid (STMG)}

We solve~\eqref{eq:lin-sys} by FGMRES iterations \cite{Saad1993} with right-sided \(hp\) space-time multigrid preconditioning. This approach was developed and analyzed in \cite{MaMuBa25} for the linear Stokes problem. Geometric refinement and coarsening of the spatial and temporal mesh is referred to as $h$-multigrid, while the refinement and coarsening of the polynomial degree \(k,r\in\mathbb{N}\) is referred to as $p$-multigrid. Here, $h$-multigrid is applied to the space variables only, since we solve~\eqref{eq:GloNS} row-wise, which amounts to a time stepping process. In the sequel, we restrict ourselves to presenting only the differences to \cite{MaMuBa25} and innovations of the preconditioner that are made for the nonlinear Navier--Stokes problem and the embedding of the FGMRES method into the Newton iterations. The preconditioner exploits the tensor product structure of the discrete solution space~\eqref{eq:GDS} and is implemented in a matrix-free form in the deal.II library \cite{kronbichlerGenericInterfaceParallel2012} in order to enhance its efficiency. The FGMRES iterations for~\eqref{eq:lin-sys} are built on the evaluation of the Jacobian matrix computed in~\eqref{eq:JacobianLocal0}, whereas a surrogate of $\boldsymbol{\mathcal J}_n(\cdot)$ is applied in the smoother of the STMG method (cf.~\Cref{subsec:VankaSmth}). To present the STMG preconditioner, we need notation. For further details, we also refer to \cite{MaMuBa25}.

\paragraph{Multilevel hierarchies and discrete spaces}

Let $\{\mathcal{T}_{s}\}_{s=0}^{S}$ be a quasi-uniform family of nested triangulations of the spatial domain $\Omega$, with characteristic mesh sizes $h_s$ satisfying $h_s \lesssim \frac{h_{s-1}}{2}$ and $h_0 = \mathcal O(1)$. This defines a hierarchy of nested spaces
\[
\boldsymbol H_{s}^{k,r+1}\coloneq Y_\tau^k(I)\otimes \boldsymbol V_s^{r+1}(\Omega)\,,
\qquad
H_{s}^{k,r}\coloneq Y_\tau^k(I)\otimes Q_s^{r}(\Omega)\,.
\]
For brevity, we let $r\ge k$ as well as $k=2^K$ and $r=2^R$ for some $K,R\in \N$.

\paragraph{Grid transfer operators}
By the mapping 
\begin{equation}
\label{eq:MatGeomProl}
P^h_{s-1\to s}:\{\boldsymbol V^{r+1}_{s-1}(\Omega),Q^r_{s-1}(\Omega)\}\to\{\boldsymbol V^{r+1}_{s}(\Omega),Q^r_{s}(\Omega)\}
\end{equation}
we denote the respective canonical embedding of the spacial finite element spaces into their refinement within the mesh hierarchy. For $\{\boldsymbol v_{\tau h},p_{\tau h}\}\in Y_\tau^k(I_n) \otimes \boldsymbol V^{r+1}_{s-1}(\Omega) \times Y_\tau^k(I_n) \otimes Q^r_{s}(\Omega)$, the matrix representation $\boldsymbol P^{n}_{s-1\to s}\in \R^{(k+1)\cdot M_s,(k+1)\cdot M_{s-1}}$ of the geometric prolongation for space-time functions on $I_n$ is the tensor product 
\begin{equation*}
\boldsymbol P^{n}_{s-1\to s} \coloneqq \boldsymbol E_{k+1}\otimes \boldsymbol P_{s-1\to s}^h
\end{equation*}
for the vector representations of $\{\boldsymbol v_{\tau h},p_{\tau h}\}$ on $I_n$, where $\boldsymbol P_{s-1\to s}^h\in \R^{M_{s},M_{s-1}}$ is the matrix representation of the prolongation \eqref{eq:MatGeomProl} in space from $\mathcal T_{s-1}$ to $\mathcal T_s$ for $\boldsymbol V_{s-1}^{r+1}(\Omega)$ or $ Q_{s-1}^{r}(\Omega)$, respectively, with $M_s$ denoting the dimension of the finite element space.

Polynomial prolongation $P^{p}_{(k/2,r/2)\to(k,r)}= P^p_{k/2\to k} \otimes P^p_{r/2\to r}$ is defined by tensor products of the prolongations in the either variables. Its matrix representation is then defined by matrix products of tensor products of the form as in \eqref{eq:MatGeomProl}. Restrictions are chosen as the adjoint of the prolongation;  cf.~\cite{MaMuBa25}. 

\paragraph{Cycle and coarsening order}
We employ a $V$-cycle multigrid approach with $\nu_1/\nu_2$ pre-/post-smoothing steps due to its superior parallel scaling properties. Coarsening is done
\emph{firstly in the polynomial degree} ($p$), i.e.\ $(k,r)\mapsto(k/2,r/2)$ until $(1,1)$ is reached, and \textit{secondly geometrically for the space mesh} ($h$), i.e\ $(s)\mapsto(s{-}1)$. We recall that $h$-multigrid for the temporal variable is not used here. If the $p$-hierarchy of the spatial variables is larger than of the temporal variable, then $p$-coarsening steps are done firstly until the hierarchy heights in space and time coincide; cf.~\cite{MaMuBa25} for the presentation of a pseudo algorithm. The coarsest problem is solved either directly or by FGMRES iterations. Corrections are prolongated in the reverse order.

This coarsening order follows the \(hp\)-STMG construction in \cite{MaMuBa25} and preserves the tensor-product structure exploited by the patch-based smoother.
Patch definitions and variants (element/vertex-star, treatment of pressure DoFs) are specified in
Subsection~\ref{subsec:VankaSmth} and in \cite{MaMuBa25}.

\subsection{Inexact space-time Vanka smoother\label{subsec:VankaSmth}}

The smoother is essential for the performance of multigrid methods. We use a local smoother of Vanka-type \cite{vankaBlockimplicitMultigridSolution1985}. This smoother has proved its efficiency and scalability for saddle point problems \cite{john_numerical_2000,wobkerNumericalStudiesVankaType2009,ahmedAssessmentSolversSaddle2018,MaMuBa25}. For its construction in the framework of space-time finite element methods we refer to \cite{anselmannGeometricMultigridMethod2023,anselmannEnergyefficientGMRESMultigrid2024}. In particular, we use a local Vanka smoother that operates on the local slabs $S_n\coloneqq K\times I_n$, for $K\in \mathcal T_s$. All global degress of freedom linked to the element $K\in \mathcal T_d$ for all $(k+1)$ Gauss–Radau time points of $I_n$ are updated simultaneously by the local Vanka systems. In our matrix-free approach, the global Jacobian matrix~\eqref{eq:JacobianLocal0} is not explicitly built in the FGMRES iterations. Instead, a matrix-free evaluation is used; cf.~\Cref{sec:mf}. However, the Vanka smoother that works on the degrees of freedom of the local slabs $S_n$ continues to be matrix-based. To accelerate its application (cf.~\cite{MaMuBa25}), an approximation of the Jacobian on $S_n$ is used here. The construction of this inexact space-time Vanka smoother is presented now.

We consider a fixed level of the joint multigrid hierarchy of space-time polynomial order and spatial mesh refinement, introduced before. To simplify the notation, we omit the indices characterizing the respective subinterval $I_n$ and multigrid level. On subinterval $I_n$, we consider solving a linear system of the form (cf.~\eqref{eq:lin-sys})
\begin{equation*}
\boldsymbol{\mathcal J} \boldsymbol d = \boldsymbol r
\end{equation*}
with the Jacobian matrix $\boldsymbol{\mathcal J}$ being defined in~\eqref{eq:JacobianLocal0} and some residual or right-hand side vector $\boldsymbol r$. For this system, the local Vanka smoother is defined by
\begin{equation}
\label{eq:LVS}
  \boldsymbol{\mathcal S}_{S_n}(\boldsymbol d) \coloneqq  \boldsymbol{\mathcal R}_{S_n} \boldsymbol d + \omega \boldsymbol{\mathcal J}_{S_n}^{-1} \boldsymbol{\mathcal R}_{S_n} (\boldsymbol r - \boldsymbol{\mathcal J} \boldsymbol d)\,.
\end{equation}
In~\eqref{eq:LVS}, we denote by $\boldsymbol{\mathcal R}_{S_n}$ the local restriction operator that assigns to a global defect vector $\boldsymbol d$ the local block vector $\boldsymbol{\mathcal R}_{S_n} \boldsymbol d$ that contains all components of $\boldsymbol d$ that are associated with all degrees of freedom linked to the slab $S_n$. Further, the parameter $\omega>0$ is an algorithmic relaxation parameter, and $\boldsymbol{\mathcal J}_{S_n}$ denotes the local Jacobian associated with the degrees of freedom on the slab $S_n$. The local Jacobian matrix $\boldsymbol{\mathcal J}_{S_n}$ inherits the block structure of its global counterpart in~\eqref{eq:JacobianLocal0}.

In~\eqref{eq:JacobianLocal2}, the block diagonal matrix $\boldsymbol H_n'(\boldsymbol V_n)$, defined in~\eqref{eq:DefHnp}, of the convective contribution depends on the block vectors $\boldsymbol V_n^a$ of the velocity degrees of freedom at all time nodes $t_{n,a}\in I_n$, for $a=1,\ldots,k+1$. This feature is inherited by its local on $K$ counterpart $\boldsymbol H_{n,K}'(\boldsymbol V_{n,K})$. It inflates both measures, setup cost and memory, of its evaluation. To reduce this, we consider using a surrogate \(\boldsymbol{\widetilde H}_{n,K}' (\boldsymbol V_{n,K}^\star)\), that evaluates the current solution in the midpoint $t_n^\ast$ of $I_n$ only. Thus, we substitute in in the local Jacobian $\boldsymbol{\mathcal J}_{S_n}$ the contribution  $\boldsymbol H_{n,K}'(\cdot)$ by
\begin{equation}
\label{eq:DefTildeHn}
\boldsymbol{\widetilde H}_{n,K}' (\boldsymbol V_n^\star) \coloneqq \operatorname{diag}(\boldsymbol H_K'(\boldsymbol V_n^\ast),\ldots,\boldsymbol H_K'(\boldsymbol V_n^\ast))
\end{equation}
with $\boldsymbol H_K'(\cdot)$ denoting the local counterpart of $\boldsymbol H'(\cdot)$, defined in~\eqref{eq:DefHp}, and
\begin{equation}
\label{eq:DefVnAst}
\boldsymbol V_{n,K}^\ast =\sum_{a=1}^{k+1}\varphi_n^a(t_n^\star)\sum_{i=1}^{M_K^{\boldsymbol v}} V^{a}_{n,K,i} \boldsymbol\chi_i^{\boldsymbol v}
\in \boldsymbol V_h^{r+1}\,,
\end{equation}
where $M^{\boldsymbol v}_K$ denotes the number of local velocity degrees of freedom on $K$. On each slab $S_n$, the spatial core \(\boldsymbol H'_K(\boldsymbol V_n^\star)\) is thus \emph{shared} across all time nodes. For the set of velocity and pressure degrees of freedom of size \(m \simeq d(r{+}2)^d + (r{+}1)^d\), factorization and storage is thus needed only once per slab with computational cost \(\mathcal O(m^3)\) and memory usage \(\mathcal O(m^2)\), instead of \((k+1)\) times. This is essential for memory-friendly and cache-efficient implementations. Letting $\boldsymbol{\widetilde{\mathcal J}}_{S_n}$ denote the resulting surrogate of $\boldsymbol{\mathcal J}_{S_n}$ in~\eqref{eq:LVS}, assembled by using~\eqref{eq:DefTildeHn}, a single local Vanka sweep reads as
\begin{equation}\label{eq:vanka0}
  \boldsymbol d \leftarrow \boldsymbol d
  + \omega \sum_{K}
  \boldsymbol R_{S_n}^{\top} (\boldsymbol{\widetilde{\mathcal J}}_{S_n})^{-1} \boldsymbol R_{S_n} (\boldsymbol r - \boldsymbol{\mathcal J} \boldsymbol d)\,,
\end{equation}
where $\boldsymbol R_{S_n}$ is the matrix representation of $\boldsymbol{\mathcal R}_{S_n}$ in~\eqref{eq:LVS}.

\begin{remark}
The Jacobian of the nonlinear term $\boldsymbol N_{\Gamma_D}^{\boldsymbol v,c}(\boldsymbol V_n)$ in~\eqref{eq:InAlgNS1}, that is due to the boundary pairings with the convective term in~\eqref{eq:def-a-gam-tp}, is approximated similiarly by a surrogate.
\end{remark}

Finally, we analyze the surrogate~\eqref{eq:DefTildeHn}. We aim to characterize the perturbation of the performance of the local Vanka smoother by buillding the Jacobian on the approximation~\eqref{eq:DefTildeHn} to~\eqref{eq:DefHnp}. In order not to overload this work, we use some simplifying assumptions about $\boldsymbol H'(\cdot)$ and $\boldsymbol V_n^\ast$, without proving them explicitly. Their rigorous proof is left as a work for the future. The contribution $\boldsymbol N_{\Gamma_D}^{\boldsymbol v,c}(\boldsymbol V_n)$ is also neglected. Nevertheless, Lemma \ref{lem:patch-perturbation} illustrates advantageous properties of the surrogate based on~\eqref{eq:DefTildeHn} and, thereby, suggests its application. In the following, all matrix norms are spectral (operator) norms. For a given matrix $\boldsymbol A$, let $s_{\min}(\boldsymbol A)$ and $s_{\max}(\boldsymbol A)$ denote its smallest and largest singular values, respectively, i.e., $s_{\max}(\boldsymbol A)=\|\boldsymbol A\|$ and, if $\boldsymbol A$ is invertible, $s_{\min}(\boldsymbol A)=\|\boldsymbol A^{-1}\|^{-1}$.

\begin{lemma}[Characterization of local Vanka surrogate]
\label{lem:patch-perturbation}
Consider a slab $S_n$ with spatial element $K$. Let $\boldsymbol S_n \coloneqq \boldsymbol{{\mathcal J}}_{S_n}\in \R^{(k+1)\cdot (M_K^{\boldsymbol v}+M_K^p),(k+1)\cdot (M_K^{\boldsymbol v}+M_K^p)}$, with $M^{\boldsymbol v}_K+M^p_K$ denoting the total number of velocity and pressure degrees of freedom on $K$, be the exact local Vanka matrix of~\eqref{eq:LVS} on $K\in \mathcal T_h$. Let $\boldsymbol{\widetilde S}_n\in \R^{(k+1)\cdot (M_K^{\boldsymbol v}+M_K^p),(k+1)\cdot (M_K^{\boldsymbol v}+M_K^p)}$ be its surrogate based on the convective contribution~\eqref{eq:DefTildeHn}. Assume local Lipschitz continuity of the local Jacobian matrix $\boldsymbol H_K'(\cdot)\in \R^{M_K^{\boldsymbol v},M_K^{\boldsymbol v}}$ on $K$, such that
\begin{equation}
\label{eq:CondLip}
\big\|\boldsymbol H_K'(\boldsymbol U_K)-\boldsymbol H_K'(\boldsymbol V_K)\big\|
  \;\le\; L_K\,\big\|\boldsymbol U_K-\boldsymbol V_K\big\|
  \qquad \text{for }\boldsymbol U_K,\boldsymbol V_K\in \R^{M^{\boldsymbol v}_K}\,.
\end{equation}
 Let $a\in \{1,\ldots,k+1\}$. For $\boldsymbol V_{n|K}(t_n^a)$, $\boldsymbol V_{n|K}(t_n^\ast)\in \boldsymbol V^{r+1}(K)$, represented by $\boldsymbol V_{n,K}^a \in \R^{M^{\boldsymbol v}_K}$ and $\boldsymbol V_{n,K}^\ast\in \R^{M^{\boldsymbol v}_K}$, respectively, assume the approximation property that
\begin{equation}
\label{eq:ApproxProp}
\|\boldsymbol V_{n,K}^{a}-\boldsymbol V_{n,K}^\star\|\le C_K\,\tau_n\,.
\end{equation}
Then the local perturbation $\boldsymbol E_{n}\coloneq \boldsymbol{\widetilde S}_n- \boldsymbol S_n $ satisfies that
  \begin{equation}
  \label{eq:ErrSn}
  \|\boldsymbol E_{n}\|\;\le\; C\,\tau_n,\qquad
  C\coloneq c \, C_K\,L_K\,\|\boldsymbol M_n^{\tau}\|\,,
  \end{equation}
where $c$ is independent of the step sizes $k$ and $\tau_n$. Let
\begin{equation}
\label{eq:CondEn}
\| \boldsymbol E_n\| \leq  \varepsilon_n \| \boldsymbol S_n^{-1}\|^{-1}\,, \; \text{with some}\; \varepsilon_n <1\,.
\end{equation}
Then there holds the following.
\begin{enumerate}[label=(\roman*),leftmargin=*,itemsep=2pt,topsep=0pt,parsep=0pt]
\item \textbf{Nonsingularity and inverse bound.}
  The matrix $\boldsymbol{\widetilde S}_{n}$ is invertible and
  \[
  \big\|\boldsymbol{\widetilde S}_n^{-1}\big\|
  \ \le\ {(1-\varepsilon_{n})}^{-1}{\big\| \boldsymbol{S}_n^{\,-1}\big\|}\,.
  \]

\item \textbf{Approximation property.} There holds that
\[
\big\|\boldsymbol I - \boldsymbol{S}_{n}^{\,-1}
\boldsymbol{\widetilde S}_{n}\big\|  \le\ \varepsilon_{n}.
\]
Hence the spectrum and the field of values satisfy that
\[
\begin{aligned}
  \sigma\left(\boldsymbol{S}_n^{\, -1}\boldsymbol{\widetilde S}_n\right)
  \subset \{z\in\mathbb C:\ |z-1|\le \varepsilon_{n}\}\,,\qquad
  \mathcal W\left({\boldsymbol S}_n^{\, -1}\boldsymbol{\widetilde S}_n \right)
  \subset \{z\in\mathbb C:\ |z-1|\le \varepsilon_{n}\}\,.
\end{aligned}
\]

\item \textbf{Singular-value (conditioning) bounds.}
For the minimal and maximal singular values there holds that
\[
\begin{aligned}
  (1-\varepsilon_{n})\,s_{\min}(\boldsymbol{S}_{n})
  \,\le\,s_{\min}(\boldsymbol{\widetilde  S}_{n})
  \,\le\,s_{\max}(\boldsymbol{\widetilde  S}_{n})
  \,\le\,(1+\varepsilon_{n})\,s_{\max}(\boldsymbol{S}_{n}).
\end{aligned}
\]
\end{enumerate}
\end{lemma}

\begin{proof}
The Lipschitz continuity~\eqref{eq:CondLip} along with the local on $K$ approximation property~\eqref{eq:ApproxProp} directly proves the assertion~\eqref{eq:ErrSn}. The Neumann series expansion of matrices directly proves item (i) and the approximation property in item (ii). Since every eigenvalue $\lambda$ of $\boldsymbol{S}_n^{\, -1}\boldsymbol{\widetilde S}_n$ satisfies that $|\lambda|\le\| \boldsymbol{S}_n^{\, -1}\boldsymbol{\widetilde S}_n\|$, the spectrum of $\boldsymbol I -  \boldsymbol{S}_n^{\, -1}\boldsymbol{\widetilde S}_n$ lies in the ball
of radius $\varepsilon_n$ around $1$, and the field-of-values inclusion follows directly. To show item (iii), we use that  $s_{\min}(\boldsymbol A\boldsymbol B)\ge s_{\min}(\boldsymbol A)\,s_{\min}(\boldsymbol B)$ and $s_{\max}(\boldsymbol A\boldsymbol B)\le s_{\max}(\boldsymbol A)\,s_{\max}(\boldsymbol B)$. Further we have the product representation  \(\boldsymbol{\widetilde S}_n = \boldsymbol{S}_n (\boldsymbol I+\boldsymbol S_n^{-1} \boldsymbol E_n)\). By means of~\eqref{eq:CondEn}, we then get that
\begin{equation*}
s_{\max}(\boldsymbol{\widetilde S}_n) \leq (1+ \varepsilon_n) s_{\max}(\boldsymbol{S}_n) \,.
\end{equation*}
The lower bound for $s_{\min}(\boldsymbol{\widetilde S}_n)$ follows similarly.
\end{proof}


\begin{remark}[On the results of Lemma~\ref{lem:patch-perturbation}]
\begin{itemize}[leftmargin=*,itemsep=0pt,topsep=0pt,parsep=0pt]
\item If $\varepsilon_{n}=\mathcal O(\tau_n)$ by means of~\eqref{eq:ErrSn}, we conclude that the
midpoint  surrogate $\boldsymbol{\widetilde S}_{n}$ and the exact local Jacobian $\boldsymbol S_n $ are uniformly close to each other (in operator norm, spectrum, field of values, and singular values).

\item The condition~\eqref{eq:ApproxProp} controls on slab $S_n = K \times I_n$ the distance between $\boldsymbol V_{n|K}(t_n^a)\in \boldsymbol V^{r+1}(K)$ and $\boldsymbol V_{n|K}(t_n^\ast)\in \boldsymbol V^{r+1}(K)$ in the Eucledian norm of their coefficient vectors. Norm equivalence between the $L^2(K)$ norm of finite element functions and the vector norm of their coefficients in the finite element basis is enured, but $h_K$ dependent,
\begin{equation*}
  c_1 h_K^{d/2} \| \boldsymbol{V}\|\leq \| \boldsymbol v_h \|_{L^2(K)} \leq c_2 h_K^{d/2} \| \boldsymbol{V}\|\,,
\end{equation*}
for $ \boldsymbol v_h = \sum_{j=1}^{M_K^{\boldsymbol v}} V_{j} \boldsymbol \chi^{\boldsymbol v}_{K,j} \in \boldsymbol V^{r+1}(K)$ and $\boldsymbol V = (V_1,\ldots , V_{{M_K}^{\boldsymbol v}})^\top$.
For brevity, we do not incorporate this into Assumption~\eqref{eq:ApproxProp},
since a rigorous \(L^2(\Omega)\) error estimate and the induced bound on the
distance of coefficient vectors is beyond the scope of this work.
We defer such an analysis to future work.

\item In \Cref{lem:patch-perturbation}, (i) guarantees that for the patch matrix the evaluation of the convective part in the midpoint does not deteriorate the local patch solve and controls its inverse, (ii) shows that the frozen patch acts as a near-identity preconditioner for the exact patch (so the Vanka update remains effective), and (iii) quantifies that conditioning,
and thus damping properties, deviate by at most $\mathcal O(\tau_n)$, which is then kept small in practice by our rebuild triggers. Precisely, the smoother is \emph{rebuilt} using an updated \(\boldsymbol U_n^\star\) to reduce \(\varepsilon_{n}\) when convergence deteriorates. Let
\[
\rho_m \coloneq \frac{\|\boldsymbol{\mathcal R}_n(\boldsymbol U_n^{m};\boldsymbol V_{n-1})\|_{\mathcal M}}{\|\boldsymbol{\mathcal R}_n(\boldsymbol U_n^{m-1};\boldsymbol V_{n-1})\|_{\mathcal M}},
\]
and $\kappa_m$ the number of FGMRES iterations required until the inexactness test~\eqref{eq:inexact} is first met. Then, choose thresholds $\theta_N>1$, $\theta_L>1$, an absolute cap $\kappa_{\rm abs}$, and a
stagnation window $(s,\vartheta)$. We rebuild if $\rho_m\ge \theta_N\rho_{m-1}$ (Newton
deterioration), or $\kappa_m\ge \max(\theta_L\kappa_{m-1},\kappa_{\rm abs})$ (linear
deterioration), or if FGMRES residuals stagnate with $\|r_{j+1}\|/\|r_j\|\ge \vartheta$ for $s$ consecutive inner iterations.

\end{itemize}
\end{remark}

\subsection{Matrix-free operator evaluation}\label{sec:mf}
In this work, linear operators are evaluated without the explicit formation and
storage of system matrices. For this, we rely on the matrix-free multigrid
framework in the \texttt{deal.II}
library~\cite{africa_dealii_2024,kronbichlerGenericInterfaceParallel2012,munchEfficientDistributedMatrixfree2023,fehnHybridMultigridMethods2020}.
All matrix-vector products \(\boldsymbol{Y}=\boldsymbol S\,\boldsymbol X\) for solving
Problem~\ref{Prob:LocAlgNS} are computed via global accumulation of element-wise
operations,
\[
\boldsymbol{S}\,\boldsymbol X
=
\sum_{c=1}^{n_c}
\boldsymbol{R}_{c,\mathrm{loc\text{-}glob}}^\top\,
\boldsymbol{S}_c\,
\boldsymbol{R}_{c,\mathrm{loc\text{-}glob}}\,\boldsymbol X ,
\qquad
\boldsymbol{S}_c
=
\boldsymbol{B}_c^\top\,\boldsymbol{D}_c\,\boldsymbol{B}_c,
\]
where \(\boldsymbol{R}_{c,\mathrm{loc\text{-}glob}}\) maps local degrees of
freedom to global indices, \(\boldsymbol{B}_c\) contains shape function and gradient
evaluations, and \(\boldsymbol{D}_c\) encodes quadrature weights and material/
flux coefficients. Sum-factorization reduces multi-dimensional kernels to
products of one-dimensional operations; vectorization further accelerates the
evaluation. These techniques are used for the spatial operators
in Problem~\ref{Prob:LocAlgNS} and their Jacobians~\eqref{eq:JacobianLocal0}--\eqref{eq:JacobianLocal1}. The temporal matrices in~\eqref{eq:InAlgNS} are precomputed
as in~\cite{margenbergSpaceTimeMultigridMethod2024a,MaMuBa25}. Products such as
\((\boldsymbol M^{\tau} \otimes \boldsymbol A_h)\boldsymbol V\) are evaluated
by computing \(\boldsymbol A_h \boldsymbol V^a\) once per temporal DoF \(a\),
followed by small block multiplications with the temporal matrices; this extends
to all Kronecker products in~\eqref{eq:InAlgNS}.

\paragraph{Navier--Stokes: matrix-free application of the slab Jacobian}
For Navier--Stokes, we apply the (state-dependent) slab Jacobian
$\boldsymbol{\mathcal J}_n(\boldsymbol U_n)$ from~\eqref{eq:JacobianLocal0}--\eqref{eq:JacobianLocal1}
in a matrix-free fashion.
We fully reuse the slab block notation~\eqref{eq:DefSubvec_0}.
The only state dependence enters through the velocity-velocity block
$\boldsymbol{\mathcal J}_n^{1,1}(\boldsymbol U_n)$, and, by~\eqref{eq:DefHnp},
this dependence is \emph{block diagonal in time}. Consequently, all state-dependent
spatial kernels are evaluated independently for each temporal basis index and are
then combined by dense multiplications with the temporal matrices
$\boldsymbol K_n^\tau,\boldsymbol M_n^\tau$ from~\eqref{eq:DefKMC}.

Concretely, given an increment $(\boldsymbol V_n,\boldsymbol P_n)$ with subvectors 
$\boldsymbol V_n^{a}\in\mathbb R^{M^{\boldsymbol v}}$, $\boldsymbol P_n^{a}\in\mathbb R^{M^p}$
($a=1,\dots,k+1$), the product
$(\boldsymbol Y_{n,\boldsymbol v},\boldsymbol Y_{n,p})=\boldsymbol{\mathcal J}_n(\boldsymbol U_n)\,(\boldsymbol V_n,\boldsymbol P_n)$
is computed as follows:
\begin{enumerate}[leftmargin=*,itemsep=2pt,topsep=2pt]
  \item \emph{Column-wise spatial products (matrix-free).}
  For each $a=1,\dots,k+1$, compute the spatial actions that appear in~\eqref{eq:JacobianLocal2}--\eqref{eq:JacobianLocal3},
  \[
    \boldsymbol q^{(a)} \coloneq \boldsymbol M_h\,\boldsymbol V_n^{a},\qquad
    \boldsymbol d^{(a)} \coloneq \boldsymbol B\,\boldsymbol V_n^{a},\qquad
    \boldsymbol g^{(a)} \coloneq \boldsymbol B^\top\,\boldsymbol P_n^{a},
  \]
  and evaluate the state-dependent velocity contribution of
  $\boldsymbol{\mathcal J}_n^{1,1}(\boldsymbol U_n)$ at the same temporal block,
  i.e.,
  \[
    \boldsymbol w^{(a)} \coloneq {\Bigl[\boldsymbol{\mathcal J}_n^{1,1}(\boldsymbol U_n)\Bigr]}^{a,a}\,\boldsymbol V_n^{a},
  \]
  where this ``diagonal block'' comprises the viscous part, linearized convection,
  and boundary terms (e.g.\ Nitsche/outflow), with coefficients evaluated on-the-fly
  from the current Newton state $\boldsymbol U_n$.

  \item \emph{Dense temporal mixing.}
  Accumulate for each row index $i=1,\dots,k+1$:
  \begin{align*}
    \boldsymbol Y_{n,\boldsymbol v}^{\,i}
    \mathrel{+}= \sum_{a=1}^{k+1} (\boldsymbol K_n^\tau)_{i a}\,\boldsymbol q^{(a)}
                 + \sum_{a=1}^{k+1} (\boldsymbol M_n^\tau)_{i a}\,\bigl(\boldsymbol w^{(a)}+\boldsymbol g^{(a)}\bigr),\qquad
    \boldsymbol Y_{n,p}^{\,i}
    \mathrel{+}= \sum_{a=1}^{k+1} (\boldsymbol M_n^\tau)_{i a}\,\boldsymbol d^{(a)}.
  \end{align*}
\end{enumerate}
The coupling induced by the jump matrix $\boldsymbol C_n^\tau$ (see~\eqref{eq:DefC})
is implemented as a separate, purely temporal update with $\boldsymbol M_h$.
The midpoint surrogate is \emph{only} used inside the Vanka smoother to enable reuse
of patch factorizations and does not modify the outer Jacobian application.

\section{\label{sec:experiments}Numerical experiments}
We assess the proposed monolithic Navier-Stokes Newton-Krylov solver in the
fully discrete setting of \Cref{sec:stfem}, using spatial
$\mathbb{Q}_{r+1}/\mathbb{P}_r^{\mathrm{disc}}$ pairs and a DG$(k)$
discretization in time. The nonlinear systems are solved by the inexact
Newton-Krylov method of \Cref{sec:nonlinear-solver} (Eisenstat-Walker
forcing, Armijo backtracking), where FGMRES is preconditioned by a single
$V$-cycle of the \(hp\) space-time multigrid (STMG) with Vanka smoothing
(\Cref{sec:mg-framework}).

Our primary goal is robustness of both the \emph{outer} Newton iterations and
the \emph{inner} linear iterations with respect to mesh size \(h\), polynomial
degree \(p\), and moderate Reynolds numbers. We assess robustness by (i) the
total number of Newton steps and (ii) the average number of FGMRES iterations
per Newton step. Specifically, we call the method \textbf{\(h\)/\(p\)-robust in
Newton} if the number of Newton steps required to meet \eqref{eq:inexact}
remains essentially bounded as \(h\!\downarrow 0\) and \(p\!\uparrow\), and
\textbf{\(h\)/\(p\)-robust in FGMRES} if the FGMRES iteration count \emph{per
Newton step} remains essentially bounded under the same refinement. We call it
\textbf{\(\mathrm{Re}\)-robust} if, over the target range
\(\mathrm{Re}\in[1,10^4]\), both Newton and FGMRES iteration counts remain
controlled. We report the average number of Newton steps $\overline{n}_{\text{NL}}$ and the average FGMRES iterations per Newton step $\overline{n}_{\text{L}}$ versus~\(\nu\). Such robustness is crucial for linear
complexity in the number of global degrees of freedom and mitigates memory
pressure from the Arnoldi basis in FGMRES. In our experiments, the method
remains within memory limits.

All tests were executed on the HSUper cluster (Helmut Schmidt University) with
571 nodes, each with two Intel Xeon Platinum~8360Y CPUs (36 cores per CPU) and
\SI[scientific-notation=false,round-precision=0]{256}{\giga\byte} RAM. In our
experiments the number of MPI processes always match the physical cores. The source code is available~\cite{margenberg_monolithic_2026}.
Unless stated otherwise, we use the following setting for the Newton. The
relative and absolute residual tolerances for Newton are $10^{-8}$ and
$10^{-12}$. The parameters for the Armijo line search are an initial step of
$\lambda_0=1$, an Armijo constant $c=10^{-4}$, a backtracking factor $\tau=0.5$,
at most $5$ backtracks, and a minimum step $\alpha_{\min}=10^{-3}$. We employ
the nonmonotone variant with window size $M=5$. For the \emph{Eisenstat-Walker
  forcing} we choose the initial $\eta$ as $\eta_{0}=0.4$
$\eta\in [\eta_{\min},\,\eta_{\max}]$ with $\eta_{\min}=10^{-3}$ and
$\eta_{\max}=0.8$, and set $c_\eta=0.5$ and $\theta=1.5$. The FGMRES solver is
steered completely by the Eisenstat-Walker forcing. However, we impose a limit
of 50 iterations which we do not hit in any of our experiments. The penalty parameters for Nitsche's method must be sufficiently large for correct enforcement of the boundary conditions. Here, we choose $\gamma_1=\gamma_2=10$.
\subsection{\label{sec:conv-ns}Convergence test}
\begin{table}[htb]
  \caption{Errors for
    $\mathbb{Q}_{r+1}^2/\mathbb{P}_{r}^{\mathrm{disc}}/\mathrm{DG}(r)$
    discretizations of the Navier-Stokes system
    for~\eqref{eq:conv-test-v} with $\nu=10^{-2}$}\label{tab:conv-ns}
  \vspace*{-2ex}
  \begin{subcaptionblock}{\textwidth}
    \caption{Calculated velocity and pressure errors in the space-time $L^2$-norm with eoc.}
    \setlength{\tabcolsep}{7pt} \centering \footnotesize
    \begin{tabular}{l|llll|llll}
      \toprule
      &\multicolumn{4}{c|}{$r=3$}&\multicolumn{4}{c}{$r=4$}\\
      $h$ & $e^{\boldsymbol v}_{L^2(L^2)}$     &eoc&$e^{p}_{L^2(L^2)}$&eoc&$e^{\boldsymbol v}_{L^2(L^2)}$ &eoc&$e^{p}_{L^2(L^2)}$&eoc\\
      \midrule
      ${2}^{-1}$ &\num{2.79971e-02} &   - &\num{1.34061e-02}&   - &\num{1.72931e-03} &   - &\num{1.04883e-03} &    -\\
      ${2}^{-2} $&\num{1.27387e-03} &4.46 &\num{7.37713e-04}&4.18 &\num{1.76103e-04} &3.30 &\num{1.15068e-04} & 3.19\\
      ${2}^{-3} $&\num{4.92974e-05} &4.69 &\num{4.78427e-05}&3.95 &\num{3.26625e-06} &5.75 &\num{3.63770e-06} & 4.98\\
      ${2}^{-4}$& \num{1.62100e-06} &4.93 &\num{2.98764e-06}&4.00 &\num{5.38812e-08} &5.92 &\num{1.13359e-07} & 5.00\\
      ${2}^{-5}$& \num{5.12810e-08} &4.98 &\num{1.86291e-07}&4.00 &\num{8.54891e-10} &5.98 &\num{3.52747e-09} & 5.01\\
      \bottomrule
    \end{tabular}
  \end{subcaptionblock}

  \vspace*{-1ex}
  \begin{subcaptionblock}{\textwidth}
    \caption{Calculated velocity errors in the space-time $L^2(H^{1})$-norm and divergence with eoc.}
    \setlength{\tabcolsep}{7pt} \centering\footnotesize
    \begin{tabular}{l|llll|llll}
      \toprule
      &\multicolumn{4}{c|}{$r=3$}&\multicolumn{4}{c}{$r=4$}\\
      $h$ & $e^{\boldsymbol v}_{L^2(H^1)}$     &eoc&$e^{\boldsymbol{\nabla \cdot v}}_{L^2(L^2)}$&eoc&$e^{\boldsymbol v}_{L^2(H^1)}$ &eoc&$e^{\boldsymbol{\nabla \cdot v}}_{L^2(L^2)}$&eoc\\
      \midrule
      ${2}^{-1}$ &\num{6.945345e-01}&  - &\num{5.8049e-01} &   - &\num{5.230669e-02}&   - &\num{4.7054e-02}&     - \\
      ${2}^{-2}$ &\num{5.483490e-02}&3.66&\num{4.9633e-02} &3.55 &\num{8.947813e-03}&2.55 &\num{8.3109e-03}&  2.50 \\
      ${2}^{-3}$ &\num{4.003204e-03}&3.78&\num{3.8147e-03} &3.70 &\num{3.248534e-04}&4.78 &\num{3.1410e-04}&  4.73 \\
      ${2}^{-4}$ &\num{2.539962e-04}&3.98&\num{2.4741e-04} &3.95 &\num{1.070734e-05}&4.92 &\num{1.0484e-05}&  4.90 \\
      ${2}^{-5}$ &\num{1.588098e-05}&4.00&\num{1.5565e-05} &3.99 &\num{3.399991e-07}&4.98 &\num{3.3442e-07}&  4.97 \\
      \bottomrule
    \end{tabular}
  \end{subcaptionblock}
\end{table}
\begin{table}[htb]
  \centering
  \caption{Average number of Newton
    iterations $\overline{n}_{\text{NL}}$ per timestep (left) and average number of FGMRES iterations
    per Newton step $\overline{n}_{\text{L}}$ (right) until convergence for
    polynomial degrees $r$ and number of refinements $c$ with  \(\mathbb{Q}_{r+1}^2/\mathbb{P}_{r}^{\text{disc}}/\text{DG}(r)\)
    discretization.}\label{tab:iter-ns}
  \vspace*{-2ex}
  \begin{subcaptionblock}{\textwidth}
    \setlength{\tabcolsep}{5.5pt}
    \caption{Results for $\nu=10^{-2}$, i.\,e.\ $\mathrm{Re}=100$.}
  \begin{minipage}{0.47\textwidth}
      \centering\scriptsize
      \begin{tabular}{ccccccc}
        \toprule
        \diagbox[innerleftsep=0pt,innerrightsep=0pt, height=1.5em, width=1.5em]{$r$}{$c$}  &  1  &  2  &  3  &  4  &  5  &  6  \\
        \midrule
        1& 6.00 &5.25 &4.75 &4.00 &4.00 &3.34\\
        2& 6.00 &5.00 &4.00 &3.09 &3.00 &3.00\\
        3& 7.00 &5.00 &4.18 &4.00 &4.00 &4.00\\
        4& 6.75 &5.00 &4.00 &4.00 &4.00 &4.00\\
        5& 6.75 &5.00 &4.00 &3.72 &3.00 &3.00\\
        6& 6.75 &4.75 &3.81 &3.00 &4.00 &4.00\\
        \bottomrule
      \end{tabular}
  \end{minipage}
  \hspace{.3cm}
  \begin{minipage}{0.47\textwidth}
    \centering\scriptsize
      \begin{tabular}{ccccccc}
        \toprule
        \diagbox[innerleftsep=0pt,innerrightsep=0pt, height=1.5em, width=1.5em]{$r$}{$c$} &  1  &  2  &  3  &  4  &  5  &  6  \\
        \midrule
        1&  5.71& 4.79& 3.83& 3.75& 3.00& 2.60\\
        2&  5.96& 4.55& 3.86& 3.05& 2.91& 2.58\\
        3&  7.43& 6.85& 4.96& 4.31& 3.23& 4.61\\
        4&  8.04& 6.35& 4.50& 3.25& 2.50& 2.00\\
        5& 10.56& 7.30& 6.14& 4.19& 4.00& 3.19\\
        6& 11.11& 8.21& 5.16& 4.64& 3.00& 2.74\\
        \bottomrule
      \end{tabular}
    \end{minipage}
  \end{subcaptionblock}
  \begin{subcaptionblock}{\textwidth}
    \setlength{\tabcolsep}{5.5pt}
    \caption{Results for $\nu=10^{-4}$, i.\,e.\ $\mathrm{Re}=1\cdot10^{3}$.}
    \begin{minipage}{0.47\textwidth}
      \centering\scriptsize
      \begin{tabular}{ccccccc}
        \toprule
        \diagbox[innerleftsep=0pt,innerrightsep=0pt, height=1.5em, width=1.5em]{$r$}{$c$}  &  1  &  2  &  3  &  4  &  5  &  6  \\
        \midrule
        1 &5.75 &5.88 &5.00 &4.00 &4.00 &4.00\\
        2 &6.50 &5.88 &4.38 &4.00 &3.00 &3.00\\
        3 &7.25 &5.75 &4.00 &4.00 &4.00 &3.93\\
        4 &7.00 &5.75 &4.68 &3.97 &3.34 &3.30\\
        5 &7.75 &5.75 &4.88 &4.00 &3.95 &3.25\\
        6 &7.75 &5.75 &4.00 &3.94 &3.00 &3.00\\
        \bottomrule
      \end{tabular}
  \end{minipage}
  \hspace{.3cm}
  \begin{minipage}{0.47\textwidth}
    \centering\scriptsize
      \begin{tabular}{ccccccc}
        \toprule
        \diagbox[innerleftsep=0pt,innerrightsep=0pt, height=1.5em, width=1.5em]{$r$}{$c$} &  1  &  2  &  3  &  4  &  5  &  6  \\
        \midrule
        1 & 6.65 &5.17 &4.51 &4.00 &3.27 &2.62\\
        2 & 5.96 &4.85 &4.23 &4.22 &3.66 &2.86\\
        3 & 8.48 &6.59 &6.89 &5.41 &4.16 &2.76\\
        4 & 7.75 &6.22 &5.07 &4.76 &4.01 &3.04\\
        5 & 9.97 &7.09 &6.21 &5.95 &4.60 &3.79\\
        6 &10.03 &7.09 &6.31 &5.33 &5.04 &4.00\\
        \bottomrule
      \end{tabular}
    \end{minipage}
  \end{subcaptionblock}
  \begin{subcaptionblock}{\textwidth}
    \setlength{\tabcolsep}{5.5pt}
    \caption{Results for $\nu=10^{-4}$, i.\,e.\ $\mathrm{Re}=1\cdot10^{4}$.}
    \begin{minipage}{0.47\textwidth}
      \centering\scriptsize
      \begin{tabular}{ccccccc}
        \toprule
        \diagbox[innerleftsep=0pt,innerrightsep=0pt, height=1.5em, width=1.5em]{$r$}{$c$}  &  1  &  2  &  3  &  4  &  5  &  6  \\
        \midrule
        1 &6.50 & 6.38 & 5.00 & 4.00 & 4.00 & 4.00\\
        2 &7.38 & 6.88 & 5.00 & 4.00 & 3.00 & 3.00\\
        3 &7.75 & 7.00 & 4.00 & 4.00 & 3.09 & 3.00\\
        4 &8.25 & 7.13 & 5.00 & 4.00 & 3.97 & 3.63\\
        5 &8.50 & 7.50 & 4.94 & 4.00 & 3.83 & 3.42\\
        6 &8.50 & 7.50 & 4.81 & 4.00 & 3.69 & 3.59\\
        \bottomrule
      \end{tabular}
  \end{minipage}
  \hspace{.3cm}
  \begin{minipage}{0.47\textwidth}
    \centering\scriptsize
      \begin{tabular}{ccccccc}
        \toprule
        \diagbox[innerleftsep=0pt,innerrightsep=0pt, height=1.5em, width=1.5em]{$r$}{$c$} &  1  &  2  &  3  &  4  &  5  &  6  \\
        \midrule
        1 &6.35 & 4.93 & 4.78 & 4.11 & 3.39 & 2.79\\
        2 &5.75 & 4.43 & 4.19 & 4.52 & 4.85 & 4.28\\
        3 &7.30 & 6.95 & 6.63 & 5.88 & 5.35 & 4.75\\
        4 &6.80 & 6.25 & 5.30 & 5.04 & 4.15 & 3.50\\
        5 &8.35 & 7.00 & 6.61 & 6.14 & 5.25 & 4.44\\
        6 &9.75 & 7.33 & 6.30 & 5.59 & 5.06 & 4.22\\
        \bottomrule
      \end{tabular}
    \end{minipage}
\end{subcaptionblock}
\end{table}
As a first test case, we consider a model problem on the space-time domain $\Omega\times I = [0,1]^2\times [0, 1]$ with prescribed solution given for the velocity $\boldsymbol v \colon \Omega\times I \to \mathbb{R}^2$ and pressure
$p \colon \Omega\times I \to \mathbb{R}$ by
\begin{subequations}\label{eq:conv-test}
\begin{align}
  \label{eq:conv-test-v}
  \boldsymbol v(\mathbf{x},\,t) &= \sin(t) \begin{pmatrix} \sin^2(\pi x) \sin(\pi y) \cos(\pi y) \\
                                        \sin(\pi x) \cos(\pi x) \sin^2(\pi y) \end{pmatrix},\\
\label{eq:conv-test-p} p(\mathbf{x},\,t) &= \sin(t) \sin(\pi x) \cos(\pi x)
  \sin(\pi y) \cos(\pi y)\,.
\end{align}
\end{subequations}
We choose the kinematic viscosity as $\nu \in \{10^{-2}, 10^{-4}, 10^{-4}\}$ and choose
the external force $\mathbf{f}$ such that the solution~\eqref{eq:conv-test}
satisfies~\eqref{eq:NSE}. The initial velocity is
prescribed as zero and homogeneous Dirichlet boundary conditions are imposed on
$\partial\Omega$ for all times
\[
  \boldsymbol v = \mathbf{0}\text{ on }\Omega\times \{0\},\quad
  \boldsymbol v = \mathbf{0}, \text{ on } \partial \Omega\times (0, T]\,.
\]
The space-time mesh $\mathcal{T}_{h}\otimes\mathcal{M}_{\tau}$ is a uniform
triangulation of the space-time domain $\Omega\times I$. We use discretizations
with varying polynomial degrees $r \in \{1,\, 3,\dots,\,6\}$ in space and
$k=r$ in time to test the convergence.

Table~\ref{tab:conv-ns} shows the findings of our convergence study for
$r\in\{3,4\}$. The expected orders of convergence agree with the experimental
rates. A complete set of results is provided in \Cref{fig:conv-ns} in
Appendix~\ref{Sec:CP}. The $L^2(0,T;L^2(\Omega)^d)$ velocity error does not
always attain the ideal rate $r+2$ due to the temporal polynomial degree $k$,
whereas the $L^2(0,T;H^1(\Omega)^d)$ norm consistently exhibits the optimal
order $r+1$, which supports the choice $k=r$ in time for the tests below.

Table~\ref{tab:iter-ns} reports the nonlinear and linear iteration counts for
these experiments: the number of Newton steps, and the average FGMRES iterations
per Newton step when preconditioned by a single $V$-cycle \(hp\)~STMG. The method
shows excellent $h$- and $p$-robustness of both the Newton method and the inner
Krylov solver. As the viscosity decreases (moderately increasing the Reynolds
number), we observe no significant growth in outer Newton iterations and inner
FGMRES iterations, indicating good $\mathrm{Re}$-robustness of the solver {in the studied range}. The
Armijo globalization and Eisenstat-Walker forcing are effective. We note that the rebuild of the Vanka-smoother was always triggered by slow Newton convergence, and the linear solver was not affected by the inexact smoother.

While polynomial coarsening ($r\to r-1$) can reduce FGMRES
iterations, wall-clock gains may be limited by slowly shrinking local block
sizes; we therefore use degree halving ($r\to r/2$), as also
advocated in our previous work~\cite{MaMuBa25}.
We use a single smoothing step on all levels, i.e., $\nu_1=\nu_2=1$.
Additional smoothing can reduce the number of FGMRES iterations and may
decrease the need for smoother rebuilds. In our computations, however, we
typically rebuild the smoother only once per time step, so larger
$\nu_1,\nu_2$ would increase the cost of each smoother application; see
\eqref{eq:LVS}. Keeping the number of smoothing steps small is therefore critical
for overall performance (cf.~\cite{margenbergSpaceTimeMultigridMethod2024a}). In
the present section, we achieve excellent $h$-, $p$-, and
$\mathrm{Re}$-robustness. We revisit this trade-off in the next section on
large-scale simulations and assess whether increased smoothing improves
robustness in practice.

\subsection{Lid-driven cavity flow}
\begin{table}[htb]
  \centering
  \caption{Discretization sizes for refinement levels $c$ and polynomial degrees $r$:
  number of global space-time elements (\# st-elements) and total number of unknowns $N_{\text{dof}}$.}
  \label{tab:st-sizes}
  \footnotesize
  \sisetup{scientific-notation=true, round-precision=3}
  \begin{tabular}{l S[table-format=1.3e2] | S[table-format=1.3e2]  S[table-format=1.3e2]  S[table-format=1.3e2] }
    \toprule
    & & \multicolumn{3}{c}{$N_{\text{dof}}$}\\
    \mc{$c$} & \multicolumn{1}{c|}{\# st-elements} & \mc{$r=2$} & \mc{$r=3$} & \mc{$r=4$}\\
    \midrule
    4 & 1048576     & 302520576      & 927534080       & 2252187807\\
    5 & 16777216    & 4708913664     & 14531434496     & 36035004919\\
    6 & 268435456   & 74307412992    & 236159008240    & 576560078710\\
    7 & 4294967296  & 1180701050880  & 3778544131834   & 9224961259361\\
    \bottomrule
  \end{tabular}
\end{table}
\begin{table}[htb]
  \centering
  \caption{Iteration counts for the nonlinear solve for different numbers of smoothing steps $n_{\text{sm}}$
    and kinematic viscosities $\nu$: average number of Newton iterations $\overline{n}_{\text{NL}}$ and
    average number of FGMRES iterations per Newton step $\overline{n}_{\text{L}}$ for polynomial degrees $r$
    and refinements $c$. Problem sizes (\# st-elements, $N_{\text{dof}}$) are given in Table~\ref{tab:st-sizes}.}
  \label{tab:iters-ns}
  \footnotesize
  \sisetup{round-precision=2}
  \begin{subcaptionblock}{\textwidth}\setlength{\tabcolsep}{5.5pt}
  \centering
  \caption{Linear and nonlinear iteration counts for $\nu=4\cdot 10^{-4}$ and $n_{\text{sm}}=\nu_1=\nu_2\in\{1,2\}$ (left/right)}
  \begin{tabular}{l | S[table-format=1.2] S[table-format=1.2] | S[table-format=1.2] S[table-format=1.2] | S[table-format=1.2] S[table-format=1.2]}
    \toprule
    & \multicolumn{2}{c|}{$r=2$} & \multicolumn{2}{c|}{$r=3$} & \multicolumn{2}{c}{$r=4$}\\
    \multicolumn{1}{c|}{$c$} & \mc{$\overline{n}_{\text{NL}}$} & \multicolumn{1}{c|}{$\overline{n}_{\text{L}}$}
            & \mc{$\overline{n}_{\text{NL}}$} & \multicolumn{1}{c|}{$\overline{n}_{\text{L}}$}
            & \mc{$\overline{n}_{\text{NL}}$} & \mc{$\overline{n}_{\text{L}}$}\\
    \midrule
    4 & 5.00 & 2.61 & 5.00 & 3.42 & 5.11 & 3.55\\
    5 & 5.00 & 2.15 & 5.00 & 2.83 & 4.99 & 2.95\\
    6 & 4.99 & 1.91 & 4.96 & 2.17 & 4.99 & 2.88\\
    7 & 4.98 & 1.90 & 4.93 & 2.11 & 4.96 & 2.58\\
    \bottomrule
  \end{tabular}
  \hfill
  \begin{tabular}{l | S[table-format=1.2] S[table-format=1.2] | S[table-format=1.2] S[table-format=1.2] | S[table-format=1.2] S[table-format=1.2]}
    \toprule
    & \multicolumn{2}{c|}{$r=2$} & \multicolumn{2}{c|}{$r=3$} & \multicolumn{2}{c}{$r=4$}\\
    \multicolumn{1}{c|}{$c$} & \mc{$\overline{n}_{\text{NL}}$} & \multicolumn{1}{c|}{$\overline{n}_{\text{L}}$}
            & \mc{$\overline{n}_{\text{NL}}$} & \multicolumn{1}{c|}{$\overline{n}_{\text{L}}$}
            & \mc{$\overline{n}_{\text{NL}}$} & \mc{$\overline{n}_{\text{L}}$}\\
    \midrule
    4 & 5.00 & 1.66 & 5.00 & 2.08 & 4.99 & 2.51\\
    5 & 5.00 & 1.52 & 4.99 & 1.71 & 4.94 & 1.76\\
    6 & 4.96 & 1.38 & 4.95 & 1.57 & 4.89 & 1.72\\
    7 & 4.90 & 1.26 & 4.88 & 1.47 & 4.81 & 1.70\\
    \bottomrule
  \end{tabular}
  \end{subcaptionblock}

  \begin{subcaptionblock}{\textwidth}\setlength{\tabcolsep}{5.5pt}
  \centering
  \caption{Linear and nonlinear iteration counts for $\nu=2\cdot 10^{-4}$ and $n_{\text{sm}}=\nu_1=\nu_2\in\{1,2\}$ (left/right)}
  \begin{tabular}{l | S[table-format=1.2] S[table-format=1.2] | S[table-format=1.2] S[table-format=1.2] | S[table-format=1.2] S[table-format=1.2]}
    \toprule
    & \multicolumn{2}{c|}{$r=2$} & \multicolumn{2}{c|}{$r=3$} & \multicolumn{2}{c}{$r=4$}\\
    \multicolumn{1}{c|}{$c$} & \mc{$\overline{n}_{\text{NL}}$} & \multicolumn{1}{c|}{$\overline{n}_{\text{L}}$}
            & \mc{$\overline{n}_{\text{NL}}$} & \multicolumn{1}{c|}{$\overline{n}_{\text{L}}$}
            & \mc{$\overline{n}_{\text{NL}}$} & \mc{$\overline{n}_{\text{L}}$}\\
    \midrule
    4 & 4.99 & 2.65 & 4.99 & 3.47 & 5.09 & 3.60\\
    5 & 4.99 & 2.18 & 4.99 & 2.87 & 4.99 & 2.99\\
    6 & 4.98 & 1.94 & 4.96 & 2.20 & 4.99 & 2.92\\
    7 & 4.98 & 1.93 & 4.93 & 2.14 & 4.96 & 2.62\\
    \bottomrule
  \end{tabular}
  \hfill
  \begin{tabular}{l | S[table-format=1.2] S[table-format=1.2] | S[table-format=1.2] S[table-format=1.2] | S[table-format=1.2] S[table-format=1.2]}
    \toprule
    & \multicolumn{2}{c|}{$r=2$} & \multicolumn{2}{c|}{$r=3$} & \multicolumn{2}{c}{$r=4$}\\
    \multicolumn{1}{c|}{$c$} & \mc{$\overline{n}_{\text{NL}}$} & \multicolumn{1}{c|}{$\overline{n}_{\text{L}}$}
            & \mc{$\overline{n}_{\text{NL}}$} & \multicolumn{1}{c|}{$\overline{n}_{\text{L}}$}
            & \mc{$\overline{n}_{\text{NL}}$} & \mc{$\overline{n}_{\text{L}}$}\\
    \midrule
    4 & 4.99 & 1.67 & 4.99 & 2.10 & 4.99 & 2.59\\
    5 & 4.99 & 1.53 & 4.99 & 1.69 & 4.93 & 1.77\\
    6 & 4.97 & 1.35 & 4.95 & 1.51 & 4.89 & 1.70\\
    7 & 4.89 & 1.25 & 4.86 & 1.42 & 4.77 & 1.63\\
    \bottomrule
  \end{tabular}
  \end{subcaptionblock}
\end{table}
\begin{table}[htb]
  \centering
  \caption{Throughput $\theta$ and wall time for different values of $\nu$, $n_{\text{sm}}$, $r$ and $c$.
  Problem sizes (\# st-elements, $N_{\text{dof}}$) are reported in Table~\ref{tab:st-sizes}.}
  \label{tab:lid-throughput-stokes}
  \footnotesize
\begin{subcaptionblock}{\textwidth}
  \centering
  \caption{Throughput for $\nu=4\cdot 10^{-4}$ and $n_{\text{sm}}=\nu_1=\nu_2\in\{1,2\}$ (left/right)}
  \sisetup{round-precision=3}
  \begin{tabular}{lS[table-format=2.3e2]S[table-format=2.3e2]S[table-format=2.3e2]}
    \toprule
    \mc{$c$} & \mc{$r=2$} & \mc{$r=3$} & \mc{$r=4$}\\
    \midrule
    4 & 1200478.4762  & 986738.3830   & 800066.7165\\
    5 & 3293793.9251  & 2750421.7358  & 2049772.7485\\
    6 & 11422483.2527 & 7909755.3575  & 6199570.7388\\
    7 & 35778819.7236 & 25530703.5935 & 17999924.4085\\
    \bottomrule
  \end{tabular}
\hfill
  \begin{tabular}{lS[table-format=2.3e2]S[table-format=2.3e2]S[table-format=2.3e2]}
    \toprule
    \mc{$c$} & \mc{$r=2$} & \mc{$r=3$} & \mc{$r=4$}\\
    \midrule
    4 & 1080430.6286  & 900518.5243   & 662408.1785\\
    5 & 3092989.1967  & 2892403.3631  & 1715952.6152\\
    6 & 10129663.1485 & 6320292.1776  & 5147857.8456\\
    7 & 28111929.7829 & 18892720.6592 & 14192248.0913\\
    \bottomrule
  \end{tabular}
\end{subcaptionblock}

\begin{subcaptionblock}{\textwidth}
  \centering
  \caption{Throughput for $\nu=2\cdot 10^{-4}$ and $n_{\text{sm}}=\nu_1=\nu_2\in\{1,2\}$ (left/right)}
  \sisetup{round-precision=3}
  \begin{tabular}{lS[table-format=2.3e2]S[table-format=2.3e2]S[table-format=2.3e2]}
    \toprule
    \mc{$c$} & \mc{$r=2$} & \mc{$r=3$} & \mc{$r=4$}\\
    \midrule
    4 & 1311109.7137  & 1077672.1986  & 873797.6268\\
    5 & 3597336.6417  & 3003889.4707  & 2238671.5086\\
    6 & 12560414.6369 & 8697741.5303  & 6817184.7861\\
    7 & 39343179.6751 & 28074125.0402 & 19793114.0718\\
    \bottomrule
  \end{tabular}
\hfill
  \begin{tabular}{lS[table-format=2.3e2]S[table-format=2.3e2]S[table-format=2.3e2]}
    \toprule
    \mc{$c$} & \mc{$r=2$} & \mc{$r=3$} & \mc{$r=4$}\\
    \midrule
    4 & 1208951.1299  & 1007637.9349  & 741203.6411\\
    5 & 3460909.6457  & 3236463.5186  & 1920070.3849\\
    6 & 12684775.1779 & 7914526.2933  & 6446356.4545\\
    7 & 35202899.0387 & 23658231.3299 & 17772108.8715\\
    \bottomrule
  \end{tabular}
\end{subcaptionblock}

\end{table}

\begin{figure}[htbp]
\includegraphics{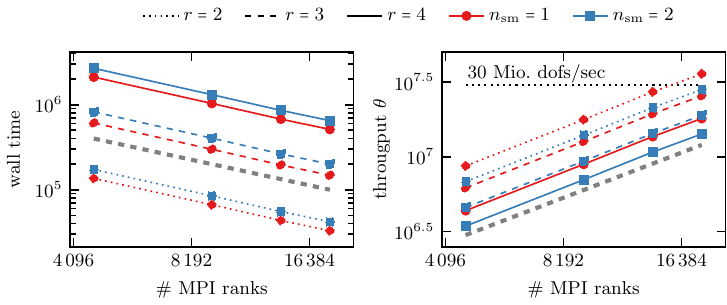}
\caption{\label{fig:lid-strong-scale}Strong scaling test results for the STMG
  algorithm with varying numbers of smoothing steps (anchored to the $c=7$ case at $18432$ MPI ranks, cf.\ Table~\ref{tab:lid-throughput-stokes}).
  The left plot shows the time to solution over the number of MPI processes. The dashed gray lines
  indicate the optimal scaling. The right plot depicts the degrees of freedom (dofs) processed per second
  over the number of MPI processes.}
\end{figure}

\begin{figure}[htbp]
\pgfplotsset{
  /pgfplots/bar cycle list/.style={/pgfplots/cycle  list={%
       {myred,fill=myred,mark=none},%
       {mygreen,fill=mygreen,mark=none},%
       {myblue,fill=myblue,mark=none},%
       {gray!80!black,fill=gray!80!black,mark=none},%
    }
  },
}
\centering
\begin{tikzpicture}[font=\small]
  \begin{axis}[title={$\mathbb{Q}_{2}/\mathbb{P}_{1}^{\text{disc}}/DG(1)$},
    every axis title/.style={at={(0.1,1.25)}},
        xbar stacked,
        width=.99\textwidth,
        height=.2\textwidth, enlarge y limits=0.5, xmin=-1,xmax=101,
        legend style={/tikz/every even column/.append style={column sep=0.3cm},
          at={(0.5,2.2)}, anchor=north,legend columns=-1,draw=none,},
        symbolic y coords={1,2},
        ytick=data,
        ylabel={$n_{\text{sm}}$},
        every node near coord/.style={
          text=white,font={\footnotesize\bfseries\sffamily},check for zero/.code={
            \pgfkeys{/pgf/fpu=true}
            \pgfmathparse{\pgfplotspointmeta-4}
            \pgfmathfloatifflags{\pgfmathresult}{-}{\pgfkeys{/tikz/coordinate}}{}
            \pgfkeys{/pgf/fpu=false}
          },
          check for zero},
        nodes near coords={\pgfmathprintnumber[assume math mode=true,fixed zerofill,precision=1]{\pgfplotspointmeta}}]

    \addplot+[draw=none,bar width=0.4cm,fill=myred]
      plot coordinates {(49.75,2) (55.87,1)};
    \addlegendentry{Rebuild Vanka}

    \addplot+[draw=none,bar width=0.4cm,fill=mygreen]
      plot coordinates {(9.53,2) (7.55,1)};
    \addlegendentry{Apply Vanka}

    \addplot+[draw=none,bar width=0.4cm,fill=myblue]
      plot coordinates {(18.19,2) (19.52,1)};
    \addlegendentry{MG w/o Vanka}

    \addplot+[draw=none,bar width=0.4cm,fill=gray!80!black]
      plot coordinates {(22.53,2) (17.06,1)};
    \addlegendentry{Other}
  \end{axis}
\end{tikzpicture}

\vspace{2mm}

\begin{tikzpicture}[font=\small]
  \begin{axis}[title={$\mathbb{Q}_{3}/\mathbb{P}_{2}^{\text{disc}}/DG(2)$},
        every axis title/.style={at={(0.1,1.25)}},
        xbar stacked,
        width=.99\textwidth,
        height=.2\textwidth, enlarge y limits=0.5, xmin=-1,xmax=101,
        symbolic y coords={1,2},
        ytick=data,
        ylabel={$n_{\text{sm}}$},
        every node near coord/.style={
          text=white,font={\footnotesize\bfseries\sffamily},check for zero/.code={
            \pgfkeys{/pgf/fpu=true}
            \pgfmathparse{\pgfplotspointmeta-4}
            \pgfmathfloatifflags{\pgfmathresult}{-}{\pgfkeys{/tikz/coordinate}}{}
            \pgfkeys{/pgf/fpu=false}
          },
          check for zero},
        nodes near coords={\pgfmathprintnumber[assume math mode=true,fixed zerofill,precision=1]{\pgfplotspointmeta}}]

    \addplot+[draw=none,bar width=0.4cm,fill=myred]
      plot coordinates {(72.32,2) (77.36,1)};

    \addplot+[draw=none,bar width=0.4cm,fill=mygreen]
      plot coordinates {(8.62,2) (6.03,1)};

    \addplot+[draw=none,bar width=0.4cm,fill=myblue]
      plot coordinates {(5.06,2) (1.42,1)};

    \addplot+[draw=none,bar width=0.4cm,fill=gray!80!black]
      plot coordinates {(14.0,2) (15.19,1)};
  \end{axis}
\end{tikzpicture}
\begin{tikzpicture}[font=\small]
  \begin{axis}[title={$\mathbb{Q}_{4}/\mathbb{P}_{3}^{\text{disc}}/DG(3)$},
        every axis title/.style={at={(0.1,1.25)}},
        xbar stacked,
        width=.99\textwidth,
        height=.2\textwidth, enlarge y limits=0.5, xmin=-1,xmax=101,
        symbolic y coords={1,2},
        ytick=data,
        xlabel={Relative Time [\%]},
        ylabel={$n_{\text{sm}}$},
        every node near coord/.style={
          text=white,font={\footnotesize\bfseries\sffamily},check for zero/.code={
            \pgfkeys{/pgf/fpu=true}
            \pgfmathparse{\pgfplotspointmeta-4}
            \pgfmathfloatifflags{\pgfmathresult}{-}{\pgfkeys{/tikz/coordinate}}{}
            \pgfkeys{/pgf/fpu=false}
          },
          check for zero},
        nodes near coords={\pgfmathprintnumber[assume math mode=true,fixed zerofill,precision=1]{\pgfplotspointmeta}}]
    \addplot+[draw=none,bar width=0.4cm] plot coordinates
      {(85.10,1) (79.80,2)};
    \addplot+[draw=none,bar width=0.4cm] plot coordinates
      {(7.90,1) (9.70,2)};
    \addplot+[draw=none,bar width=0.4cm] plot coordinates
      {(4.00,1) (6.00,2)};
    \addplot+[draw=none,bar width=0.4cm] plot coordinates
      {(3.00,1) (4.50,2)};
    \end{axis}
\end{tikzpicture}
\caption{\label{fig:lid-strong-scale-rel}Relative time spent in dominant parts of
  the Navier-Stokes solve on $\num[scientific-notation=false,round-precision=0]{18432}$ MPI ranks.
  The cell-wise Vanka smoother is dominated by the expensive
  rebuild and application of the Vanka smoother.
  Data shown for $c=7$, $r\in\{2,\,3,\,4\}$, $n_{\text{sm}}\in\{1,\,2\}$.}
\end{figure}
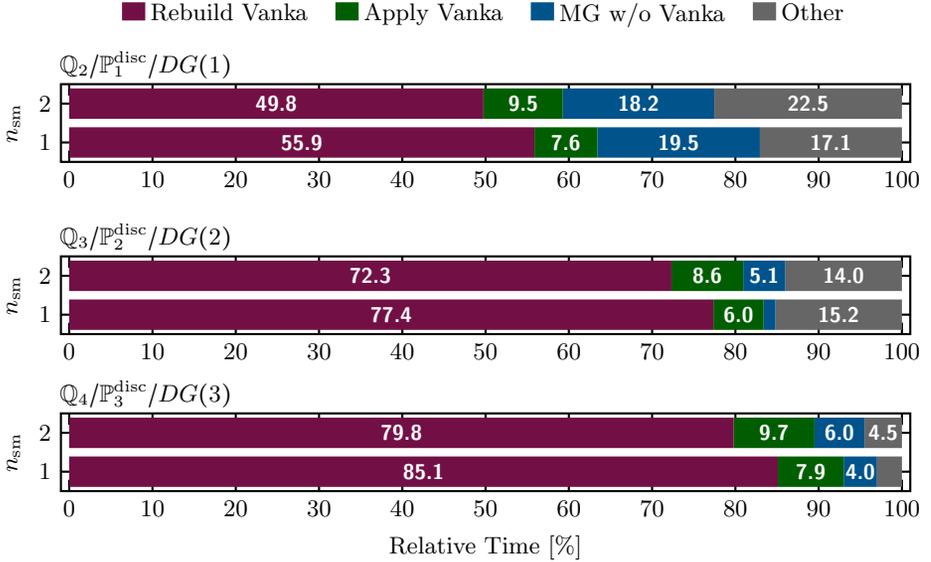
We now study the benchmark problem of lid-driven cavity flow. The space-time mesh
$\mathcal{T}_{h}\times\mathcal{M}_{\tau}$ is a uniform triangulation of the
space-time domain $\Omega\times I=[0,\,1]^3\times [0,\,8]$, refined globally $c$
times. A Dirichlet profile $\mathbf v_{D}$ is prescribed at the upper boundary
$\Gamma_{D}=[0,\,1]^{2}\times \{1\}\subset \partial \Omega$ by
\[
\mathbf v_{D}(x,\,y,\,z,\,t)=\sin\!\left(\tfrac{\pi}{4} t\right)
\quad\text{on }\Gamma_{D}\times [0,\,8],
\]
and no-slip conditions are imposed on
$\Gamma_{\text{wall}}=\partial \Omega \setminus\Gamma_{D}$. We consider
$\nu \in \{2\cdot 10^{-4}, 4\cdot 10^{-4}\}$ and degree
$r\in\{2,\,3,\,4\}$ in space with $k=r$ in time. For the strong scaling test in
\Cref{fig:lid-strong-scale} we fix $c=7$, i.e.\ $\num{2097152}$ space elements
and $\num[scientific-notation=false,round-precision=0]{2048}$ time elements; cf.\ \Cref{tab:st-sizes}.

\paragraph{Iterations and robustness}
The nonlinear solve is stable across all configurations: the average Newton
counts remain essentially constant (about five iterations) for both viscosities
and all degrees, see \Cref{tab:iters-ns}. The average FGMRES iterations per
Newton step are likewise well controlled and decrease mildly with refinement,
indicating mesh-independent preconditioning. Increasing the smoothing steps from
$n_{\text{sm}}=1$ to $n_{\text{sm}}=2$ consistently reduces linear iterations
for both viscosities and all degrees (most noticeably for larger $r$), i.e.\
additional smoothing improves \(p\)-robustness, although some degree dependence
remains visible in \Cref{tab:iters-ns}.

\paragraph{Time to solution and throughput}
We quantify performance by the throughput
$\theta(n_{\text{sm}},c,r)=N_{\text{dof}}(c,r)/W_{\text{total}}(n_{\text{sm}},c,r)$.
Wall times and throughputs are summarized in \Cref{tab:lid-throughput-stokes}.
For fixed $(\nu,n_{\text{sm}},r)$, the throughput increases strongly with
refinement, reaching the multi-$10^{7}$ dofs/s regime for $r=2$ at $c=7$ (and
correspondingly high absolute values for $r=3,4$). While increasing
$n_{\text{sm}}$ reduces linear iterations (\Cref{tab:iters-ns}), it typically
increases wall time (\Cref{tab:lid-throughput-stokes}) because the additional
Vanka work dominates the cost per iteration. The dependence on $\nu$ is
comparatively mild: iteration counts change only marginally
(\Cref{tab:iters-ns}), and differences in time to solution are primarily
attributable to the shifted balance between setup and application cost in the
preconditioner (\Cref{tab:lid-throughput-stokes}).

\paragraph{Strong scaling and bottlenecks}
\Cref{fig:lid-strong-scale} shows near-optimal strong scaling for $c=7$ with
$n_{\text{sm}}\in\{1,2,4\}$ and $r\in\{2,3,4\}$: throughput grows almost linearly with MPI ranks. The
runtime breakdown in \Cref{fig:lid-strong-scale-rel} explains the remaining
performance limitations: the Vanka smoother dominates, in particular the costly
preconditioner rebuild and the Vanka application, whereas the remainder of the
multigrid work becomes comparatively small. Consequently, the iteration
reductions obtained by increasing $n_{\text{sm}}$ do not translate into shorter
time to solution, and the dominance of rebuild/apply becomes more pronounced at
higher polynomial degree (\Cref{fig:lid-strong-scale-rel},
\Cref{tab:lid-throughput-stokes}). Overall, the \(hp\) STMG method yields robust
iteration counts and excellent scalability (\Cref{tab:iters-ns},
\Cref{fig:lid-strong-scale}), but the method is expensive due to the
computational complexity of the smoother (\Cref{fig:lid-strong-scale-rel}).

\section{\label{sec:conclusions}Conclusions}
We presented a monolithic matrix-free \(hp\) space-time multigrid method
(\(hp\)-STMG) for tensor-product space-time finite element discretizations of
the incompressible Navier--Stokes equations, based on mapped inf-sup stable
pairs \(\mathbb Q_{r+1}/\mathbb P_{r}^{\mathrm{disc}}\) in space and DG\((k)\) in
time. Fully coupled nonlinear systems are solved by Newton--GMRES preconditioned
with \(hp\)-STMG, combining geometric and polynomial coarsening in space and
time. The numerical experiments demonstrate robust nonlinear convergence,
\(h\)- and \(p\)-robust Krylov convergence of the \(hp\)-STMG preconditioner, and
solver performance that remains stable over the investigated range of Reynolds
numbers, while achieving near-optimal strong scaling up to \(\num{18432}\) MPI
ranks and high throughput on problems with more than \(10^{12}\) degrees of
freedom.

The performance analysis identifies the dominant bottleneck: the space-time
cell-wise Vanka smoother. We proposed an effective approximation {via coefficient patch models with single time point evaluation} and an inexact application that preserves the
robust iteration behavior, yet rebuild and apply remain expensive and dominate
time to solution at higher order. The practical results nicely confirm the theory and their outcomes:
{freezing the coefficients of the nonlinear terms on the local Vanka patch by single time point evaluation} introduces a perturbation controlled by the time-step
\(\tau\), so the inexact patch model remains consistent and improves as
\(\tau\to 0\). {Tailored temporal quadrature of the nonlinear terms with inexactness for triple products of discrete functions enables efficient time integration, while implying only higher-order quadrature errors.} It therefore does not reduce
the DG\((k)\) convergence order. Future work should therefore target cheaper smoothing strategies, for example
block-diagonal or approximate factorization variants (e.,g.\ diagonal
Vanka~\cite{john_numerical_2000}), and the incorporation of temporal decoupling
ideas into the local space-time systems (e.g.\ along the lines of
\cite{munchStageParallelFullyImplicit2023}), in order to reduce both compute and
memory costs without compromising robustness.

\section*{Acknowledgments}
Computational resources (HPC cluster HSUper) have been provided by the project
hpc.bw, funded by dtec.bw - Digitalization and Technology Research Center of the
Bundeswehr. dtec.bw is funded by the European Union - NextGenerationEU.

\appendix
\setcounter{table}{0}
\renewcommand{\thetable}{\Alph{section}.\arabic{table}}
\setcounter{figure}{0}
\renewcommand{\thefigure}{\Alph{section}.\arabic{figure}}

\section{Definition of the assembly matrices and vectors}
\label{App:Assembly}

Here we summarize the matrices and vectors that are assembled from the bilinear and linear forms in~\eqref{eq:DNSE}. This is done for the tensor product spaces in~\eqref{eq:GDS} and their bases introduced in~\eqref{eq:BasYk} and~\eqref{eq:BasVhQh}, respectively.

\paragraph{Time matrix assembly}
For the temporal basis induced by~\eqref{eq:BasYk}, we define $\boldsymbol K_n^\tau$, $\boldsymbol M_n^\tau$, $\boldsymbol C_n^\tau\in \R^{k+1,k+1}$ by
\begin{subequations}\label{eq:DefKMC}
  \begin{align}
    \label{eq:DefKM}
    (\boldsymbol K_n^\tau)_{ab} &\coloneq \int_{t_{n-1}}^{t_n} \partial_t\varphi_n^b\,\varphi_n^a\d t + \varphi_n^b(t_{n-1}^+)\,\varphi_n^a(t_{n-1}^+)\,, \quad   (\boldsymbol M_n^\tau)_{ab} \coloneq \int_{t_{n-1}}^{t_n} \varphi_n^b\,\varphi_n^a\d t \,, \\
    \label{eq:DefC}
    (\boldsymbol C_n^\tau)_{ab} &\coloneq \begin{cases}
      \varphi_{n-1}^b(t_{n-1})\,\varphi_n^a(t_{n-1}^+)\,, & \text{for } n>1\,,\\
      \varphi_n^a(t_{n-1}^+)\,\delta_{b,k+1}\,, & \text{for } n=1\,,
    \end{cases}
  \end{align}
\end{subequations}
with the Kronecker symbol $\delta_{\alpha,\beta}$. By the exactness of formula~\eqref{eq:GF} for all polynomials of degree less or equal than $2k$, $\boldsymbol M_n^\tau$ is diagonal with positive entries $w_{n,\mu}\coloneq\tfrac{\tau_n}{2}\hat\omega_\mu$.

\paragraph{Space matrix assembly (bilinear forms)}
Using the bases in~\eqref{eq:BasVhQh}, we define $\boldsymbol M_h$, $\boldsymbol A_h \in\mathbb R^{M^{\boldsymbol v}\times M^{\boldsymbol v}}$, $\boldsymbol B_h \in\mathbb R^{M^p\times M^{\boldsymbol v}}$ and $\boldsymbol M_h^p \in\mathbb R^{M^p\times M^p}$ by
\begin{subequations}\label{eq:DefMAB}
  \begin{align}
    (\boldsymbol M_h)_{ij} &\coloneq \int_\Omega \boldsymbol\chi^{\boldsymbol v}_j\cdot\boldsymbol\chi^{\boldsymbol v}_i\d \boldsymbol x\,, & (\boldsymbol A_h)_{ij} &\coloneq \int_\Omega \boldsymbol\nabla\boldsymbol\chi^{\boldsymbol v}_j\cdot \boldsymbol\nabla\boldsymbol\chi^{\boldsymbol v}_i\d \boldsymbol x\,,\\
    (\boldsymbol B_h)_{\ell j} &\coloneq - \int_\Omega (\boldsymbol\nabla\cdot\boldsymbol\chi^{\boldsymbol v}_j)\,\chi^p_{\ell}\d  \boldsymbol x\,,& (\boldsymbol M_h^p)_{\ell m} &\coloneq \int_\Omega \chi^p_m\,\chi^p_{\ell}\d  \boldsymbol x\,.
  \end{align}
\end{subequations}

\paragraph{Space matrix assembly (boundary terms)}
For the bilinear boundary pairings in~\eqref{eq:def-b-gam} and~\eqref{eq:def-a-gam-tp}, that are due to either the natural boundary conditions arising from integration by parts or the application of Nitsche's method, we define $\boldsymbol G_{\Gamma_D}^{\boldsymbol v} \in\mathbb R^{M^{\boldsymbol v},M^{\boldsymbol v}}$, $\boldsymbol G_{\Gamma_D}^{p} \in\mathbb R^{M^{\boldsymbol v},M^p}$, $\boldsymbol M_{\Gamma_D} \in\mathbb R^{M^{\boldsymbol v}, M^{\boldsymbol v}}$ and $\boldsymbol M_{\Gamma_D}^{\boldsymbol n} \in\mathbb R^{M^{\boldsymbol v},M^{\boldsymbol v}}$ by
\begin{subequations}
  \label{eq:DefNitscheSpatial1}
  \begin{align}
    (\boldsymbol G_{\Gamma_D}^{\boldsymbol v})_{ij} &\coloneq \int_{\Gamma_D} (\partial_n\boldsymbol\chi^{\boldsymbol v}_j)\cdot  \boldsymbol\chi^{\boldsymbol v}_i\d o\,, & (\boldsymbol G_{\Gamma_D}^{p})_{ij} &\coloneq \int_{\Gamma_D} (\boldsymbol \chi^{\boldsymbol v}_j \cdot \boldsymbol n) \boldsymbol \chi_i^p\d o\,,\\
    (\boldsymbol M_{\Gamma_D})_{ij} &\coloneq \int_{\Gamma_D} \boldsymbol\chi^{\boldsymbol v}_j\cdot\boldsymbol\chi^{\boldsymbol v}_i\d o\,, & (\boldsymbol M_{\Gamma_D}^{\boldsymbol n})_{ij} &\coloneq \int_{\Gamma_D} (\boldsymbol\chi^{\boldsymbol v}_j\cdot\boldsymbol  n) \,  (\boldsymbol\chi^{\boldsymbol v}_i\boldsymbol n) \d o\,,
  \end{align}
\end{subequations}
and, finally, $\boldsymbol N^{\boldsymbol v;b,r}_{\Gamma_D}(\gamma)\in\mathbb R^{M^{\boldsymbol v},M^{\boldsymbol v}}$ by
\begin{equation}
  \label{eq:DefNitscheSpatial10}
    \boldsymbol N^{\boldsymbol v;b,r}_{\Gamma_D}(\gamma) \coloneq -\nu\,(\boldsymbol G_{\Gamma_D}^{\boldsymbol v} + (\boldsymbol G_{\Gamma_D}^{\boldsymbol v})^{\top}) + \nu\,\frac{\gamma_1}{h_{\Gamma_D}} \, \boldsymbol M_{\Gamma_D}+ \frac{\gamma_2}{h_{\Gamma_D}} \, \boldsymbol M_{\Gamma_D}^{\boldsymbol n}\,,
\end{equation}

\paragraph{Vectors assembly (linear forms)}

For the linear forms in~\eqref{eq:DNSE} we put
\begin{equation}
  \label{eq:DefFn}
  \boldsymbol F_{n} \coloneq (\boldsymbol F_{n}^{1},\ldots,\boldsymbol F_{n}^{k+1})^\top \in \R^{(k+1)\cdot M^{\boldsymbol v}}\,, \quad \text{with}\; \;
  (\boldsymbol F_{n,}^{a})_i \coloneq Q_n(\langle \boldsymbol f, \varphi_n^a\, \boldsymbol{\chi}_i^{\boldsymbol{v}} \rangle)
\end{equation}
for $a=1,\ldots,k+1$ and $n=1\ldots,N$. Further, we let
\begin{equation}
  \label{eq:DefLn}
  \boldsymbol L_{n} \coloneq (\boldsymbol L_{n}^{1},\ldots,\boldsymbol L_{n}^{k+1})^\top \in \R^{(k+1)\cdot M^{\boldsymbol v}}\,, \quad \text{with}\; \;
  (\boldsymbol L_{n,}^{a})_i \coloneq Q_n(B_\gamma(\boldsymbol g, \varphi_n^a\, \boldsymbol{\chi}_i^{\boldsymbol{v}}))\,.
\end{equation}
In~\eqref{eq:DefFn} and~\eqref{eq:DefLn}, we tacitly assume that the data $\boldsymbol f$ and $\boldsymbol g$ are sufficiently smooth in time functions such that their point-wise evaluation in time is well-defined.

\section{Algorithms}
\label{App:Algorithms}
\SetKwFunction{FGMRES}{FGMRES}
\SetKwFunction{Armijo}{Armijo}
\SetKwFunction{Rebuild}{ShouldRebuild}
\SetKwFunction{UpdSm}{UpdateSmootherFromMidpoint}
\SetKwFunction{Stag}{Stagnates}

\begin{algorithm2e}[H]
\caption{Inexact Newton-Krylov (EW + nonmonotone Armijo)}
\label{alg:ink}
\DontPrintSemicolon
\KwIn{Initial $\boldsymbol X_0$, residual $\mathcal R(\cdot)$, Jacobian $\mathcal J(\cdot)$}
\KwIn{Tolerances $(\mathrm{abs}_{\rm tol},\mathrm{rel}_{\rm tol},\mathrm{maxit})$;\quad EW $(\eta_0,c_\eta,\theta,\eta_{\min},\eta_{\max})$}
\KwIn{Line-search $(\lambda_0,c,\tau,\alpha_{\min},M)$; Rebuild $(\theta_N,\theta_L,\kappa_{\rm abs})$; Stagn.\ $(s,\vartheta)$}
\BlankLine
$\boldsymbol X\gets \boldsymbol X_0$;\ $\boldsymbol r\gets \mathcal R(\boldsymbol X)$;\ $n_0\gets\|\boldsymbol r\|_{\mathcal M}$\;
$n_{\rm prev}\gets n_0$;\ $\eta\gets \eta_0$;\ $\rho_{\rm prev}\gets 1$;\ $\kappa_{\rm prev}\gets 1$;\ $\kappa_{\rm ref}\gets 1$;\ $k\gets 0$\;

\While{$\|\boldsymbol r\|_{\mathcal M}>\max(\mathrm{abs}_{\rm tol},\mathrm{rel}_{\rm tol}\,n_0)$ \textbf{and} $k<\mathrm{maxit}$}{
  \If{$\rho_{\rm prev}\ge \theta_N$ \textbf{or} $\kappa_{\rm prev}\ge \max(\theta_L\,\kappa_{\rm ref},\,\kappa_{\rm abs})$}{
    \UpdSm$(\boldsymbol X)$;\ $\kappa_{\rm ref}\gets \max(1,\kappa_{\rm prev})$
  }
  \lIf{$k>0$}{$\eta\gets \operatorname{clamp}\Big(c_\eta\,\eta\big(\|\boldsymbol r\|_{\mathcal M}/\max(n_{\rm prev},\varepsilon)\big)^\theta,\eta_{\min},\eta_{\max}\Big)$}
  $(\widehat{\boldsymbol X},\kappa,\{\rho_j\})\gets \FGMRES(\mathcal J(\boldsymbol X),-\boldsymbol r;\mathrm{tol}=\eta)$\tcp*[r]{exact Jacobian}
  $\alpha\gets \Armijo(\boldsymbol X,\widehat{\boldsymbol X},\boldsymbol r,\mathcal J(\boldsymbol X);\,\lambda_0,c,\tau,\alpha_{\min},M)$\;
  $\boldsymbol X\gets \boldsymbol X+\alpha\,\widehat{\boldsymbol X}$\;
  $n_{\rm prev}\gets \|\boldsymbol r\|_{\mathcal M}$;\ $\boldsymbol r\gets \mathcal R(\boldsymbol X)$\;
  $\rho_{\rm prev}\gets \|\boldsymbol r\|_{\mathcal M}/\max(n_{\rm prev},\varepsilon)$;\ $\kappa_{\rm prev}\gets \max(1,\kappa)$\;
  \lIf{\Stag$(\{\rho_j\},s,\vartheta)$}{\UpdSm$(\boldsymbol X)$}
  $k\gets k+1$\;
}
\Return $\boldsymbol X$\;
\end{algorithm2e}

\begin{algorithm2e}[H]
\caption{Armijo$(\boldsymbol X,\widehat{\boldsymbol X},\boldsymbol r,\mathcal J;\lambda_0,c,\tau,\alpha_{\min},M)$}
\label{alg:armijo}
\DontPrintSemicolon
\KwIn{$\phi(\boldsymbol Z)=\tfrac12\|\mathcal R(\boldsymbol Z)\|_{\mathcal M}^2$;\quad direction $\widehat{\boldsymbol X}$}
\KwOut{Step $\alpha$}
\BlankLine
$\alpha\gets \lambda_0$;\ $\phi_0\gets \tfrac12\|\boldsymbol r\|_{\mathcal M}^2$;\ $g_0\gets \boldsymbol r^\top \mathcal J\,\widehat{\boldsymbol X}$\;
Initialize nonmonotone window $\mathcal W\gets[\phi_0]$ (keep last $M$) and failures $(\alpha_1,\phi_1),(\alpha_2,\phi_2)\gets\emptyset$\;
\While{$\alpha>\alpha_{\min}$}{
  $\phi_t\gets \tfrac12\|\mathcal R(\boldsymbol X+\alpha\,\widehat{\boldsymbol X})\|_{\mathcal M}^2$;\quad
  $\phi_{\max}\gets (M>0)\ ?\ \max(\mathcal W)\ :\ \phi_0$\;
  \If{$\phi_t\le \phi_{\max}+c\,\alpha\,g_0$}{
    \lIf{$M>0$}{append $\phi_t$ to $\mathcal W$ (keep last $M$)}
    \Return $\alpha$
  }
  Update $(\alpha_1,\phi_1),(\alpha_2,\phi_2)$ with current fail $(\alpha,\phi_t)$\;
  $\alpha\gets \operatorname{clamp}\big(\alpha_{\rm interp},\,0.1\alpha,\,0.5\alpha\big)$;\quad
  \lIf{$\alpha_{\rm interp}\ \text{undefined}$}{$\alpha\gets \tau\,\alpha$}
}
\Return $\alpha_{\min}$\;
\end{algorithm2e}
\section{\label{Sec:CP} Convergence plots} \Cref{fig:conv-ns} shows the
convergence in various norms for all polynomial degrees and refinements
in~\Cref{sec:conv-ns}. The Newton and GMRES iteration counts are given in
Table~\ref{tab:iter-ns}.
\begin{figure}[htb]
  \centering
  \includegraphics{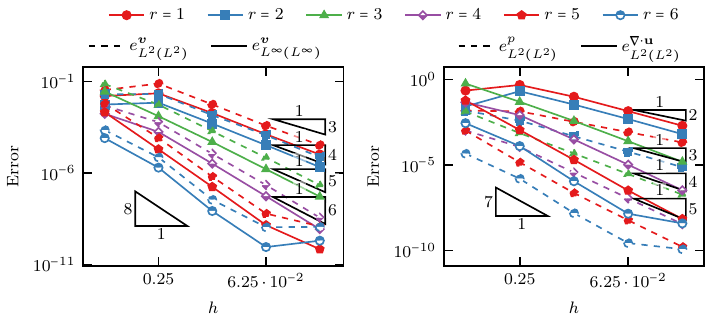}
  \caption{\label{fig:conv-ns}Calculated errors of the velocity and pressure
    in various norms (velocity: $L^2$, $L^{\infty}$ in space-time and the
    $L^2$-norm of the divergence in space-time, pressure: $L^2$ in space-time)
    for different polynomial orders. The expected orders of convergence,
    represented by the triangles, match with the experimental orders.}
\end{figure}
\bibliographystyle{siamplain}
\bibliography{references}
\end{document}